\documentclass[reqno, 11pt]{amsart}

\usepackage{amssymb}
\usepackage{amsmath}
\usepackage[usenames, dvipsnames]{color}
\usepackage{enumitem}

\newtheorem{theorem}{Theorem}[section]

\newtheorem{lemma}[theorem]{Lemma}

\theoremstyle{definition}
\newtheorem{remark}[theorem]{Remark}

\theoremstyle{definition}
\newtheorem{definition}[theorem]{Definition}

\theoremstyle{definition}

\newcommand{\mysection}[1]{\section{#1}
\setcounter{equation}{0}}

\newcommand\cbrk{\text{$]$\kern-.15em$]$}} 
\newcommand\opar{
\text{\,\raise.2ex\hbox{${\scriptstyle |}$}\kern-.34em$($}}

 \makeatletter
 \def\dashint{%
 \operatorname%
 {\,\,\text{\bf--}\kern-.98em\DOTSI\intop\ilimits@\!\!}}
 \makeatother

\newcommand{\nlimsup}{\operatornamewithlimits{\overline{lim}}}

\def\bH{\mathbb{H}}

\def\bL{\mathbb{L}}
\def\bR{\mathbb{R}}

\def\bN{\mathbb{N}}
\def\bD{\mathbb{D}}
\def\bS{\mathbb{S}}

\def\fR{\mathfrak{R}}

\def\fH{\mathfrak{H}}

\def\cF{\mathcal{F}}
\def\cK{\mathcal{K}}
\def\cH{\mathcal{H}}
\def\cP{\mathcal{P}}

\def\cR{\mathcal{R}}

\def\cV{\mathcal{V}}
\def\cW{\mathcal{W}}

\begin{document}

  \title[Support theorem]{Support theorem for an SPDE with multiplicative noise driven by a cylindrical Wiener process on the real line} 
  
	\author{Timur Yastrzhembskiy}
\email{yastr002@umn.edu}
\address{127 Vincent Hall, University of Minnesota,
 Minneapolis, MN, 55455}


  \begin{abstract}
	We prove a Stroock-Varadhan's type support theorem
	 for a  stochastic partial differential equation (SPDE) on the real line with a noise term driven by a cylindrical Wiener process on $L_2 (\bR)$.
	The main ingredients of the proof are V. Mackevi\v cius's approach to 
	support theorem for diffusion processes and N.V. Krylov's $L_p$-theory of SPDEs.
   \end{abstract}

 \maketitle
		\mysection{Introduction}
			\label{section 1}  
		Let 
		$
		  (\Omega, \cF, P)
		$ 
		be a complete probability space, 
		and  let 
		$
	 	  (\cF_t, t \geq 0)
		$
		 be an increasing filtration of $\sigma$-fields 
		$
		   \cF_t \subset \cF
		$
		 containing all $P$-null sets of $\Omega$.
		By $\cP$ we denote the predictable
		 $\sigma$-field generated by  
		  $
		   (\cF_t, t \geq 0)
		   $.

		Let
		 $
		  \bN = \{1, 2, \ldots \}
		 $, 
		$\bR$ be the real line, and  
		$
		  \bR_{+} = [0, \infty)
		$.
		Denote when it makes sense
		$$
		  	D_{x} = \frac{\partial}{\partial x},    \quad
																\partial_t = \frac{\partial }{\partial t}.
		   $$
		For a function
		 $
		   u:\bR_{+} \times \bR \to \bR
		$,
		 the temporal argument is denoted by t (or $\cdot$),
		 and the spatial argument -- by x (or $\star$).
		For a function $f:\bR \to \bR$,
		 we denote
			$$
				D f(x) = \frac{d f}{d x} (x).
			$$

		Let 
		$
		  W (t), t \in \bR_{+}
		$
		 be an $\cF_t$-adapted cylindrical Wiener process on $L_2 (\bR)$ 
		on the probability space 
		$
		 (\Omega, \cF, P)
		$
		 (see Section \ref{section 2} for the definition).
		We consider the following SPDE:
				\begin{equation}
				   \label{1.1}
			\begin{aligned}
			du (t, x)  = &  [  a (t, x) D^2_{x} u(t, x) +  b (t, x) D_x u (t, x)   + f (u, t, x) ] dt  \\
					  &  + u(t, x) dW(t),  \quad u (0, x) = u_0 (x), \, \, \, x \in \bR.
			\end{aligned}
			\end{equation}	
	Here, $a$ and $b$ are some H\"older space-valued  functions,
	$a$ is bounded from below by a positive constant,
	and
	$
	  f (u, \cdot, \star)
	$
	 is a 'zero-order' term.
	We point out that we do not assume continuity in the temporal variable for $a, b$ and $f$.	

	 In this paper we adopt N.V. Krylov's approach to parabolic SPDEs
	(see \cite{Kr_99}),
	which allows us to treat parabolic SPDEs with minimal smoothness assumptions on the coefficients and the initial data.
	Under certain conditions the equation \eqref{1.1} has a unique
	solution $u$
	 that belongs to some stochastic Banach space 
	$
	   \cH^{1/2 -\kappa }_p (T), p > 2, \kappa \in (0, 1/2]
	$
	 (see Section
	\ref{section 2}),
	which is a generalization of the parabolic counterpart of the space of Bessel potentials 
	$
	  \fH^{\gamma}_p (T).
	$
	The other approaches to the regularity theory of SPDEs can be found
	in \cite{R_90, DPZ_14, W_86}.	
	
	Our goal is to characterize the topological support of the distribution of $u$ 
	in the space 
	$
	 C^{\gamma} ([0, T], H^s_p (\bR))
	$,
	for some 
	$
	  \gamma, s \in (0, 1/2)
	$,
	 $ p > 2$,
	where 
	$
	  H^s_p (\bR)
	$
	 is the space of Bessel potentials.
	Let 
	$\cH (T)$
	 be the set of Borel functions 
	$
	  h: [0, T]\times \bR \to \bR
	$
		such that 
		$
		  \partial_t h \in B ([0, T] \times \bR) \cap L_2 ([0, T] \times \bR)
		$,
	where
	$
	  B ([0, T] \times \bR)
	$
	 is the space of bounded Borel functions.
	In Theorem \ref{theorem 2.2} we prove that the support of $u$
	coincides with the closure in the aforementioned H\"older-Bessel space
	of the set  
	$
		\fR = \{ \cR h: h \in \cH (T) \},
	$ 
	where  $\cR h$ is the  unique solution of class 
	$
	 	 \fH^{1/2 - \kappa}_p (T)
	$ 
	of the following PDE (see Definition \ref{definition 2.1} and Remark \ref{remark 2.3} (ii)):
	\begin{equation}
				\label{1.2}
	  \begin{aligned}
		\partial_t  v (t, x) &= a(t, x) D^2_{x} v (t, x) + b (t, x) D_x v (t, x) \\
																	&+ f(v, t, x) + v (t, x) \partial_t h (t, x),
																								\quad u (0, x) = u_0 (x).
	 \end{aligned}
	\end{equation}

	The support theorem for diffusion processes 
	was first proved by D. Stroock and S.R.S. Varadhan  in \cite{SV_72}.
	A different proof of this result was later given by V. Mackevi\v cius  in \cite{M_86}, where
	 the main ingredient was an approximation theorem of a Wong-Zakai type.
	This paper and V. Mackevi\v cius's proof of the Wong-Zakai theorem for diffusion processes (see \cite{M_85})
  	served as an inspiration for this article. 
	
	In case of infinitely dimensional stochastic equations the support theorems 
	were established in a number of papers.
	We will only cite the results related to parabolic SPDEs.
	In \cite{G_88} I. Gy\"ongy, adopting methods from \cite{M_85} and \cite{M_86},
	proved a support theorem 
	 for a linear SPDE on $\bR^d$ 
	  with a finite dimensional noise term. 
	In \cite{N_04} and \cite{Tw_97}   support theorems were proved
	 for SPDEs in a Hilbert space $H$ with an $H$-valued Wiener process.
	The most relevant result to ours is contained in  \cite{BMSS_95}.
	In this paper
	a support theorem  was obtained
	 for a one-dimensional nonlinear heat equation on $[0, 1]$
	 with either  Dirichlet or Neumann boundary conditions and  with a noise term 
	$
	   g(u (t, x)) dW(t)
	$.
	Here, $g$ is a sufficiently smooth function,
	 and $W$ is a cylindrical Wiener process on 
	$
	  L_2 [0, 1]
	$.
	 A similar result
	 for a one-dimensional generalized Burgers equation
	 can be found in \cite{CWM_01}.   
	In both \cite{BMSS_95} and \cite{CWM_01} the leading coefficient is equal to $1$.

	It is well-known that often it is more challenging to work with an SPDE driven by a cylindrical Wiener process
	on an unbounded domain than on a bounded interval.
	To the best of this author's knowledge,
	 there exists only one result in the literature so far 
	that is relevant to characterization of the support of the  equation \eqref{1.1}.
	In \cite{HL_18} a Wong-Zakai type theorem was proved for \eqref{1.1} 
	 with 
	$a \equiv 1, b \equiv 0 \equiv f$
	by means of M. Hairer's theory of Regularity Structures. 
	The authors showed that the sequence of Wong-Zakai type approximations converge  to the unique solution of \eqref{1.1}
	uniformly on compact subsets of 
	$
	  \bR_{+} \times \bR.
	$
	However, this result yields only one inclusion in the support theorem.

	Let us briefly describe the key steps of the proof of the main theorem of this paper.
	Our argument is similar to the one used in \cite{M_86} and \cite{G_88}.
	First, we prove an approximation theorem of a Wong-Zakai type
	 (see Theorem \ref{theorem 2.1}).
	 We replace $W$ by a 'finite-dimensional' approximation 
	$
		\sum_{k = 1}^n \phi_k (x) w^k_n (t)
	$
	and subtract a Stratonovich type correction term. 
	Here,
	 $
	   \{\phi_k, k \in \bN\}
	 $
	 is the orthonormal basis of $L_2 (\bR)$ consisting of Hermite functions,
	and	
	$
	  \{w^k, k \in \bN\}
	$
	 is a sequence of independent standard Wiener processes
	 defined by 
	$
	  w^k (\cdot) = (W(\cdot), \phi_k)_{L_2},
	 $
	and $w^k_n$ is a polygonal type approximation of $w^k$ (see Section \ref{section 2}) with some 'small' mesh size.
	To prove the approximation result
	 we use V. Mackevi\v cius's  method from \cite{M_85}, 
	which we describe below.
	We split the noise term into two parts: 
	the first one is an integral with respect to a 'regular' part 
	$
	d(w^k_n - w^k)
	$, 
	and the second one is a stochastic integral with respect to $dw^k$.
	Since
	 $
	   w^k_n - w^k
	$
	 converges to $0$ (see Lemma \ref{lemma 3.3}),
	 it makes sense to integrate by parts in the first integral.
	Then, following I. Gy\"ongy in \cite{G_88_SDE},
	we replace the solution of our approximation  scheme by its mollification
	and  integrate by parts one more time.
	As a result, we find a certain SPDE that is satisfied by the 'error' of the approximation. 
	We finish the argument by applying N.V. Krylov's $L_p$-theory of SPDEs.
	Next, one of the inclusions of the support theorem follows directly from Theorem \ref{theorem 2.1} and Portmanteau theorem.
	The other inclusion is proved
	 by  combining Theorem \ref{theorem 2.1} with Girsanov's theorem for cylindrical Wiener process.

	This author used the same method to prove a Wong-Zakai theorem and a support theorem
	for a parabolic SPDE with a finite dimensional semilinear noise term e.g. 
	$
		g((u(t, x))\, dw(t),
	$
	where $w(t)$ is a standard Wiener process 
	(see \cite{Y_18}). 
	There are three main differences between \cite{Y_18} and the present article.
	First,  some  terms that we obtain as a byproduct of integration by parts are distributions that do not belong to the $L_p$ space,
	and additional work should be done to handle them.
	Second, it can be seen from the proof of Theorem \ref{theorem 2.1} that
	 we are forced to choose a very small mesh size for $w^k_n$ because the noise is infinite dimensional.
	Third, our method fails to work if we replace $u(t, x)$ in the noise term by $g(u(t, x))$,
	where $g$ is a sufficiently smooth function such that $g(0) = 0$. 
	In particular,
	to do the integration by parts 
	we need 
	the term 
	$
	   \sum_{k = 1}^{\infty} \int_0^t (g(u(s, x)) \phi_k, \psi)_{L_2 (\bR)} \, dw^k (s), t \geq 0
	$
	 to be  a semimartingale,
		for any 
	$
	  \psi \in C^{\infty}_0 (\bR).
	$
	However, this   might not be  true, since $u$ is a solution of the equation \eqref{1.1}.
	Nevertheless, with some additional work one can use the method described above
	to prove a Wong-Zakai type theorem for a parabolic equation on $\bR$ with the noise term $g(u(t, x))\, dW(t)$.
	This will be done somewhere else.

	Finally, this author would like express his sincere gratitude to his advisor N.V. Krylov for reading a draft of this paper
	and making valuable suggestions.

		\mysection{Statement of the Main Result}
								\label{section 2}	 	
		Let $X$ be some Banach space, and $\xi$
		 be an $X$-valued random element 
		on 
		$
	 	   (\Omega, \cF, P)
		$.
		Then, by 
		$
			P \circ {\xi}^{-1}|_X
		$ 
		we denote the distribution of $\xi$,
		and by 
		$
		\text{supp} \, P \circ \xi^{-1}|_X
		$ -- the support of this probability measure.

		Let 
		$
		  C^k  = C^k (\bR), k \in \bN
		$ 
		be the space of real-valued bounded $k$ times differentiable functions with bounded derivatives up to order $k$,
		$
		  C^{\infty}_0 =  C^{\infty}_0 (\bR)
		$
		 be the space of infinitely differentiable functions with compact support.
		We denote by 
		$
		  C^{k + \alpha} =  C^{k + \alpha} (\bR), k \in \bN, \alpha \in (0, 1)
		$
		   the  H\"older space of bounded functions 
		such that derivatives up to order $k$ belong to 
		$
		  C^{\alpha} (\bR)
		$.
		For $T > 0$ finite,
		by 
		$
		  C^{k + \alpha} ([0, T], X)
		$ 
		we mean the H\"older space of  $X$-valued functions.
		For 
		$p \in [1, \infty]$,
		 we denote by 
		$
		  L_p = L_p (\bR)
		$
		  $
		   (L_p ([0, T] \times \bR))
		  $
		the space of real-valued $L_p$-integrable functions.
		Next, for $p \in (1, \infty)$,
		 we introduce spaces of Bessel potentials as follows:
		 $$
			H^{\gamma}_p  : = (1 - D^2_{x})^{- \gamma/2} L_p, 	\quad
				   										  H^{\gamma}_p (l_2) : =  (1 - D^2_{x})^{- \gamma/2} L_p (l_2).
		$$
		Here, 
		 $\gamma \in \bR$, 
		and $l_2$ is the set of all sequences of real numbers 
		$
		 h = \{h^k, k \in \bN\}
		$ 
		such that
		 $
		     |h|^2_{l_2} =  \sum_{ k = 1}^{\infty} |h^k|^2 < \infty,
		$ 
		and
		 $
		   L_p ( l_2)
		 $ 
		 is the space of sequences $h$ of 
		$L_p $
		 functions such that 
		 $
		   |h|_{l_2} \in L_p
		 $.

		 For a distribution $f$,
		 and a sequence of distributions 
		$
		  h = \{h^k, k \in \bN\}
		$, 
		we denote 
		$$
		 	|| f ||_{\gamma, p} := || ( 1 - D^2_{x})^{\gamma/2} f ||_p,
		\quad
			|| h ||_{\gamma, p} := || |(1 - D^2_{x})^{\gamma/2} h|_{l_2} ||_p,
		  $$
		where $|| \cdot ||_p$ stands for the $L_p$ norm.
		For a distribution $f$, 
		and a test function
		 $
		  g\in C^{\infty}_0
		$, 
		we denote the action of $f$ on $g$
		by   
		 $
			(f, g). 
	         $
		For any 
		$
		  f, g \in L_2
		 $, 
		their scalar product  is denoted by
		$
			(f, g)_{L_2}.
		$
		 
		The following facts about spaces
		 $
		H^{\gamma}_p, p \in (1, \infty),
		$
		  will be used in the sequel sometimes without mentioning them.
			First, for any
			 $
			   k \in \bN,
			$
			 the spaces 
			$W^k_p$
			 and
			 $H^k_p$
			coincide as sets and
			have equivalent norms.
			Here,
			 $
			   W^k_p =  W^k_p (\bR)
			 $ 
			 is  the Sobolev space of
		 	$L_p$ functions
		 	such that the generalized derivatives
			 up to order $k$ belong to $L_p$.
			  Second, 
			$$
				|| f ||_{\gamma_1, p} \leq || f ||_{\gamma_2, p}
			$$
			if 
			$
			\gamma_1 \leq \gamma_2.
			$
		Third, if 
	       $
		  \gamma \in \bR,
		$ 
		 $
			f \in H^{\gamma}_p,
		 $
		 and 
		$
		  \psi \in C^{\infty}_0,
		$
		 then
		$$
			(f, \psi) = \int_{\bR} [(1 - D^2_{x})^{\gamma/2} f (x)] [(1 - D^2_{x})^{-\gamma/2} \psi (x)] \, dx. 
		$$
		 The proof of these facts
			 and a detailed discussion of
			 $H^{\gamma}_p$
			 spaces can be found in Chapter 13 of \cite{Kr_08}.

		For any stopping time $\tau$, and 
		$
		   \gamma \in \bR, p > 1
		$,
		 we denote 	
		$
		   \opar 0, \tau \cbrk := \{ (\omega, t): 0 < t \leq \tau(\omega) \}
		$,
		   $$
			\bL_p (\tau) := L_p (\opar 0, \tau \cbrk, \cP, L_p),
		  $$
		    $$
			\mathbb{H}^{\gamma}_p (\tau) := L_p ( \opar 0, \tau \cbrk, \cP, H^{\gamma}_p),
				\quad
 				\mathbb{H}^{\gamma}_p (\tau, l_2) : = L_p ( \opar 0, \tau \cbrk, \cP, H^{\gamma}_p  (l_2)).
		   $$

		By $N (\ldots)$ we denote a constant depending 
		only on the quantities listed  inside the parenthesis.
		A constant $N$ might change from inequality to inequality.
		 In some cases, where it is clear what parameters $N$ depends on,
		we do not list them.

		The following is the definition of the stochastic Banach spaces
		 $
		   \cH^{\gamma}_p (\tau).
		$
		\begin{definition}
					\label{definition 2.1}
		Let 
		$
		  \{w^k (t), t \geq 0, k \in \bN\}
		$
		 be a sequence of
		independent $\cF_t$-adapted standard Wiener processes 
		on 
		$
		  (\Omega, \cF, P)
		$.
	   	For any
		 $\gamma \in \bR$, 
		$p \geq 2$,
		 and 
		any stopping time $\tau$,
		 we write that
		 $
		   u \in \cH^{\gamma}_p (\tau)
		$
		 if the following holds:
		\begin{enumerate} 
		 \item $u$ is a distribution-valued process, and 
			 $
				u \in  \cap_{t > 0} \bH^{\gamma}_p ( \tau \wedge t)
			$;  
		  \item
		 $ 
			D^2_{x} u \in \bH^{\gamma - 2}_p (\tau)
		$,
		    $
			 u (0, \star) \in    L_p (\Omega, \cF_0, H^{ \gamma - 2/p }_p)
		   $;
		   \item there exist
		 $
			f \in \bH^{\gamma - 2}_p (\tau)
		$
		 and
		 $
		      g = \{g^k, k \in \bN\}   \in  \bH^{\gamma - 1}_{p} (\tau, l_2)
		 $ 
		such that,
		  for any 
		 $
			\phi \in C^{\infty}_0
		$,
		 $ t \geq 0$,
		   $\omega \in \Omega$,
		\begin{equation}
				\label{1.3}
			\begin{aligned}
			(u (t \wedge \tau, \star), \phi (\star)) &
									= ( u(0, \star), \phi(\star)) + 
														\int_0^{t \wedge \tau}  (f(s, \star), \phi (\star)) \, ds\\
		 																							& + 	\sum_{k = 1}^{\infty} \int_0^{t \wedge \tau} (g^k (s, \star), \phi (\star)) \, dw^k (s).
			\end{aligned}
		\end{equation}
		\end{enumerate}
		The norm is defined in the following way:
		$$
		    || u ||_{ \cH^{\gamma}_p (\tau) }  =  ||D^2_{x} u ||_{\bH^{\gamma - 2}_p (\tau)} 
		$$
		 $$
		       + || f ||_{\bH^{\gamma - 2}_p (\tau) }   +  || g ||_{\bH^{\gamma - 1}_p (\tau, l_2) } +
				  (E   || u (0, \star) ||^p_{  \gamma - 2/p, p} )^{1/p}.
		 $$

		For $
			u \in \cH^{\gamma}_p (\tau)
			$,
		 we denote 
		 $
		    \bD u := f
		 $,
		 $
		    \bS u := g.
		$

		By 
		$
		  \fH^{\gamma}_p (T)
		$
		 we denote a subset of 
		$
		  \cH^{\gamma}_p (T)
		$
		of all functions $u$ such that 
		$
			\bS u \equiv 0,
		$
		and 
		$\bD u$
		 and 
		$u (0, \cdot)$
		 are functions independent of $\omega$.
		\end{definition}
		
			\begin{remark}
						\label{remark 2.4}
			By Remark 3.2 of \cite{Kr_99}, 
			for  any number $T > 0$,
			 the series of stochastic integrals 
			$
			\sum_{ k = 1}^{\infty} \int_0^t (g^k (s, \star), \phi (\star)) \, dw^k(s)
			$ 
			converges uniformly 
			in $t$ on $[0, T]$ in probability.
			\end{remark}

				\begin{remark}
				\label{remark 2.5}
		It was showed in Theorem 3.7 of \cite{Kr_99} that,
		for any 
		$
			\gamma \in \bR, p \geq 2
		$, 
		$
		\cH^{\gamma}_p (\tau)
		$
		 is a Banach space.
		In addition, by the same theorem 
		 if $ T > 0$ is finite,
		 and 
		$
		\tau \leq T
		$ 
		is a stopping time,
		 then,
		for any
		 $
		  v \in \cH^{\gamma}_p (\tau)
		$,
		 $$
			|| v ||_{ \mathbb{H}^{\gamma}_p (\tau) }
											 \leq N (d, T) || v ||_{\cH^{\gamma}_p (\tau) }.
		 $$	  
		It follows that, for any bounded stopping time $\tau$,
		 we may replace 
		$ 
		   ||D^2_{x} u  ||_{\mathbb{H}^{\gamma - 2}_p (\tau)}
		$
		  by
		   $
		   || u ||_{\mathbb{H}^{\gamma}_p (\tau) }
		   $
		 in the definition of the norm of 
		  $
			\cH^{\gamma}_p (\tau)
		 $
		  and obtain an equivalent norm.
		\end{remark}

		\textit{Assumptions.}
		Fix some numbers  $T, h > 0$, $\kappa \in (0, 1/2], p \geq 2$. 
	
		$(A1) (\kappa)$
		   $a (t, x)$, 
		$b ( t, x)$ 
		 are  real-valued 
		 $
		   B ([0, T] \times \bR)
		 $-measurable functions.
		For any 
		$t \in [0, T]$,	
		$
		  a(t, \star) \in C^{ 1 + 1/2 + \kappa + \eta}
		$,
		 $
		   b (t, \star) \in C^{ 1/2 + \kappa + \eta}
		 $,
		and 
		$$
			||a (t, \star) ||_{C^{ 1 +  1/2 + \kappa + \eta}} 
												+ || b (t, \star) ||_{C^{1/2 + \kappa + \eta}} \leq L,
		$$
		where $L > 0$,
		 and 
		$
		  \eta \in (0, 1/2 - \kappa)
		$ 
		are finite.
		In addition, there exists a constant 
		$\lambda > 0$ such that,
		   for all 
		$ t, x$,
		 	$$
			    \lambda   \leq a (t, x)   \leq \lambda^{-1}.
		 	$$
		
		$(A2)  (p, \kappa)$
		 $ 
		  f(u, t, x) 
		 $ 
		is a real-valued function
		defined on 
		$ 
		  \bR \times [0, T] \times \bR.
		$

		 $(i)$ For any 
		  $
		    x , u\in \bR,
		  $
		    $ f (t, x, u)$ 
		 is a Borel measurable function.

		   $(ii)$ There exists a constant $K > 0$ such that, 
			   for any  
			 $
			   t, x, u, v,
			 $
			   we have
		        $$
			     | f (u, t, x) -  f(v, t, x) | 
				\leq K | u - v |.
		         $$

		$(iii)$ 
			$
			  f (0, \cdot, \star) \in L_p( [0, T], H^{-3/2 - \kappa}_p).
			$	

		$(A3) (p, \kappa)$  
		   $
			u_0  \in  H^{1/2 - \kappa  - 2/p}_p.
		   $

		$(A4) (h)$  
		Denote 
		$
			\varkappa (x) = -1 \vee x \wedge 1, x \in \bR.
		$
		 For each $i \in \bN$,    
		   $
		    w^i (\cdot, h)
		   $ 
		is the polygonal approximation of $w^i$ with mesh size $h$ 
		defined as follows: 
		\begin{equation}
			\label{2.1}
		         w^i (t, h) := w^i (  (l-1) h ) + 1/h \, ( t -  l h ) \varkappa (w^i ( l h  ) - w^i ( (l-1) h ))
		\end{equation}
		  if
		 $
		    t \in [l h, (l+1) h),
		$
		 for some
		 $
		   l \in \mathbb{N} \cup \{0\}.
		$
		We assume here that
		 $
		  w^i (t) = 0,
		$
		 for $t \leq 0$.
		If 
		$
		   \{\gamma_n, n \in \bN\}
		$
		 is a sequence,
		then, we denote
		$
		   w^i_n (t) := w^i (t, \gamma_n).
		$

		\textit{Statement of the main result.}
		We say that 
		$
		    W (t), t \geq 0
		$ 
		is an
		 $\cF_t$-adapted
		 cylindrical Wiener process  on $L_2$
		on  
		$
		  (\Omega, \cF, P)
		$
		if the following holds:

		$(i)$ for every 
		$
		  \psi \in L_2
		$, 
		  $
		    (W(t), \psi)_{L_2}, t \geq 0
		 $
		 is an $\cF_t$-adapted standard Wiener process;
	
		$(ii)$ for any
		$t, s \geq 0$,
		and $\psi, \phi \in L_2$,
		 we have
		$$
			E (W(t), \psi)_{L_2} (W(s), \phi)_{L_2} = t \wedge s (\psi, \phi)_{ L_2}.
		$$

		The equation \eqref{1.1}  can be rewritten as follows:
		\begin{equation}
			\label{2.10}
			\begin{aligned}
			du (t, x)  = &  [ a (t, x) D^2_{x} u(t, x) +  b (t, x) D_x u (t, x)  + f (u, t, x) ] dt \\
					  &   
					  +  \sum_{k=1}^{\infty} u (t, x) \phi_k (x)  dw^k(t),  \, \, u (0, x) = u_0 (x),
		     \end{aligned}
		\end{equation}
		where \
		 $
		  \{\phi_k, k \in  \bN\}
		$ is the
		Hermite orthonormal basis of $L_2$,
		and 
		\begin{equation}
		   			\label{2.11}
		  \{w^k (\cdot)= (W(\cdot), \phi_k)_{L_2}, k \in \bN\}
		 \end{equation}
		 is a sequence of independent $\cF_t$-adapted standard Wiener processes.
		Let us recall the construction of the Hermite basis.
		First, we define the Hermite polynomials as follows:
		$$
			H_k (x) = (-1)^k e^{x^2} D^k (e^{-x^2}). 
		$$
		Then, the $k$-th member of the Hermite basis is given by
		\begin{equation}
				\label{2.4}
			\phi_k (x) = \frac{ H_k (x) e^{-x^2/2} }{ (\sqrt{\pi} 2^k k!)^{1/2} }.
		\end{equation}

		\begin{definition}
		We say that the equation \eqref{2.10} has a solution $u$
		of class 
		$
			\cH^{\gamma}_p (T)
		$ 
		if 
		  $
		     u \in \cH^{\gamma}_p (T)
		  $ 
		with
		$$
				\bD u (t, x) =a (t, x) D^2_{x} u (t, x) 
						      +   b (t, x) D_x u (t, x) +  f (u, t, x),
		$$
		  $$
			\bS u (t, x) = \{u (t, x) \phi_k (x), k \in \bN\}, 
														 \quad u (0, x) = u_0 (x).
		  $$ 	
		Recall that this implies that
		 $
			\mathbb{D} u \in \mathbb{H}^{\gamma - 2}_p (T)
		 $, 
		    $
			 \mathbb{S} u \in \mathbb{H}^{\gamma - 1}_p (T, l_2)
		    $, 
		and 
			$
			  u_0 \in L_p (\Omega, \cF_0, H^{\gamma - 2/p}_p).
			$
		\end{definition}

		 Assume that 
		$(A1) (\kappa)$,
		 $(A2) (p, \kappa)$, 
			$(A3) (p, \kappa)$ 
		hold.
	       Then, by Theorem 8.5 of \cite{Kr_99} 
		(see  Remark \ref{remark 2.2})
		the equation \eqref{2.10} 
	         has a unique solution $u$ of class 
		$
		  \cH^{1/2 - \kappa}_p (T)
		$.  
		In addition, 
		 there exists a constant 
		 $
		   N (p, \kappa, \eta, L,  K, \lambda,  T) > 0
		$
		such that the following estimate holds:
		\begin{equation}
					\label{2.14}
			  || u ||^p_{   \cH^{1/2 - \kappa}_p (T)   } \leq  
				   N  \int_0^T || f (0, t, \star) ||^p_{-3/2 - \kappa, p} \, dt
				       + N   || u_0 ||^p_{ 1/2 - \kappa - 2/p, p}.
		\end{equation}
	
		\begin{remark}
			\label{remark 2.2}
		The assumption 
		$(A2) (p, \kappa)$
		 corresponds to Assumption 8.6 of \cite{Kr_99},
		and the assumption
		 $(A3) (p, \kappa)$
		 is mentioned in the statement of Theorem 8.5 of \cite{Kr_99}. 
		However, the assumption 
		$(A1) (\kappa)$
		 is weaker than Assumption 8.5.
		Actually,
		 $(A1) (\kappa)$
		 corresponds to Assumption 5.3 and Assumption 5.5 of \cite{Kr_99}
		 with
		 $
		  n = -3/2 - \kappa
		$.
 		The conclusion of  Theorem 8.5 of \cite{Kr_99} still holds in our case,
		since  its proof is a combination of the proof of Theorem 5.1 
		(with $n = -3/2 - \kappa$)
		 and Lemma 8.4 (both are from \cite{Kr_99}).  
		\end{remark}

		\begin{remark}
			\label{remark 2.1}
		By Theorem 7.2 of \cite{Kr_99} there exists a modification of $u$, such that,
		for any $\theta$ and $\mu$ satisfying 
		$
		 1 > \mu > \theta > 2/p
		$,
		we have
		$
		  u \in C^{\theta/2 - 1/p} ([0, T], H^{1/2 - \kappa - \mu}_p)
		$,
		for all $\omega \in \Omega$. 
		Moreover,  for any stopping time $\tau \leq T$,
		$$
			E || u ||^p_{ C^{\theta/2 - 1/p} ([0, \tau], H^{1/2 - \kappa - \mu}_p) }
																  \leq N (p, \theta, \mu, T) || u ||^p_{ \cH^{1/2 - \kappa}_p (\tau) }.
		$$  
		Furthermore, if
		 $
		  \delta = 1/2  - \kappa - \mu - 1/p > 0
		$,
		then,  by the embedding theorem for $H^{s}_p$ spaces 
		(see, for example, Theorem 13.8.1 of \cite{Kr_08})
		we have
		$
		  u \in   C^{\theta/2 - 1/p} ([0, T], C^{\delta})
		$, 
		for all $\omega$.
		\end{remark}		

		Here is the statement of the main result.
		\begin{theorem}
				\label{theorem 2.2}
		Let $T > 0$,
		 and let $p > 2$, 
		  $
		   \kappa \in (0, 1/2)
		  $
		 be numbers such that 
		$
		 1/2 - \kappa > 3/p.
		$
		 We assume that
		 $
		  (A1) (\kappa)
		 $,
		  $
			(A2) (p, \kappa)
		  $,
		 $
			(A3) (p, \kappa)
		$
		 hold.
		Let 
		 $ 
		   \mathfrak{R}_{cl}
		  $
		  be the closure of $\fR$ (see Section \ref{section 1} for the definition)
		 in the space 
		$
		   \cV (T) : =C^{\theta/2 -1/p} ([0, T], H^{\varkappa}_p)
		$,
		where 
		$
		\varkappa = 1/2 - \kappa - \mu
		$,
		and
		 $\mu$ and $\theta$ are any numbers such that
		 $
		  1/2 - \kappa - 1/p > \mu > \theta > 2/p
		$.
		 Let $u$ be the unique solution of 
		\eqref{1.1} of class
		 $
		  \cH^{1/2 - \kappa}_p (T)
		$.
		Then, 
			$
			  \text{supp} \, P \circ u^{-1}|_{\cV (T)} =  \mathfrak{R}_{cl}
			$.
		\end{theorem}		

		To prove the support theorem we need an approximation result that we present below.
		
		For 
		 $\alpha, \beta \in \bR$,
		we consider the following SPDE:
		\begin{equation}
	  	  \label{2.15}
		  d v (t, x) = [a (t, x) D^2_{x} v (t, x) + b (t, x) D_x v (t, x) 
		\end{equation}
	  	$$
		 									  + f (v, t, x)]\, dt 
																+ (\alpha + \beta) \sum_{k = 1}^{\infty} v (t, x) \phi_k (x) dw^k (t),
																													\, \, v (0, x) = u_0 (x).
	 	$$  
	Also, for any  sequence 
		$
		  \{\gamma_n, n \in \bN\}
		$ 
		such that $\gamma_n > 0, n \in \bN$,
	we consider the following equation:
	 \begin{equation}
		\label{2.16}
		d v_n (t, x) = [ a (t, x) D^2_{x} v_n (t, x) +  b (t, x) D_{x} v_n (t, x)  
	\end{equation}
	     $$	
				+ f (v_n, t, x) +  \alpha \sum_{ k = 1}^n v_n (t, x) \phi_k (x) D w^k_n (t) 
	    $$
	     $$
			       - (\alpha^2/2 + \alpha \beta) \sum_{k = 1}^n v_n (t, x) \phi^2_k (x)] \, dt
	     $$
	      $$	
				+ \beta \sum_{ k = 1}^{\infty} v_n (t, x) \phi_k (x) \,  dw^k (t), 
				\, \, v_n (0, x) = u_0 (x).
	      $$
	The term 
	$$
		(\alpha^2/2 + \alpha \beta) \sum_{k = 1}^n v_n (t, x) \phi^2_k (x) \, dt
	$$
	 is akin to the so-called Stratonovich correction term.
	In fact, if
	 $\alpha = 1, \beta  = 0$, 
	then, it is exactly 
	the Stratonovich correction term of a Wong-Zakai type approximation scheme of the equation \eqref{2.10} (see Definition \ref{definition 2.2}).

	Here is the statement of the approximation theorem.
		\begin{theorem}
				\label{theorem 2.1}
		Accept the conditions of Theorem \ref{theorem 2.2}. 
		In addition, assume that either
		 $
		  \alpha = 1, \beta = 0
		$
		or 
		$
		\alpha = -1, \beta = 1,
		$
		and let $v$ be the unique solution of class 
		$
		   \cH^{1/2 - \kappa}_p (T)
		$
		of the equation \eqref{2.15}.
		Then, there exists a sequence
			 $
			    \{\gamma_n, n \in \bN\}
			  $
			such that, if we additionally assume that $(A4) (\gamma_n)$
			holds,
			and let $v_n$ be the unique solution
			 of class 
			$
			\cH^{1/2 - \kappa}_p (T)
			$
			of \eqref{2.16}
			 (see Remark \ref{remark 2.3} (i)),
			then, we have
		   \begin{equation}
					\label{2.17}
			  || v_n - v||_{ \cV (T)  }   \to 0
		   \end{equation}
		   in probability as $n \to \infty$.
		\end{theorem}

	\begin{definition}
		\label{definition 2.2}
	Under the assumptions of Theorem \ref{theorem 2.2},
	we say that \eqref{2.16} is a Wong-Zakai type approximation scheme of \eqref{2.10} 
	 if  
	  $
		 \alpha  = 1, \beta  = 0,
	$
	and  
	$
	  \{\gamma_n, n \in \bN\}
	$ 
	is a sequence such that
	$$
		|| v_n - u||_{ \cV (T) } \to 0
	$$
	in probability as $n \to \infty$.
	Here, $v_n$ is the unique solution of \eqref{2.16} of class 
	$
	   \cH^{1/2 - \kappa}_p (T),
	$
	and $u \in  \cH^{1/2 - \kappa}_p (T)$
	 is the unique solution of \eqref{2.10}. 
	\end{definition}

	\begin{remark}
		\label{remark 2.3}
	 Assume the conditions of Theorem \ref{theorem 2.2}.
	
	$(i)$ 
	Let 
	$
	  \{ \gamma_n, n \in \bN\}
	$
	 be any sequence such that $\gamma_n > 0, n \in \bN$.
	We claim that the equation  \eqref{2.16}  has a unique solution
	$v_n$
	 of class 
	$
	  \cH^{1/2 - \kappa}_p (T).
	$
	For 
	$
	  u, x \in \bR
	$, 
	and 
	 $
		t \geq 0
	$ 
	we set 
	$$
		\bar f (u, t, x)  = f (u, t, x) +  \alpha  u  \sum_{ k = 1}^n  \phi_k (x) D w^k_n (t)
				- (\alpha^2/2 + \alpha \beta)  u  \sum_{k = 1}^n   \phi^2_k (x), 
	 $$
	  $$
		      h (u, t, x)  =  \beta u, 
	  $$
	   $$	
		\bar K = \beta + K + \gamma^{-1}_n \sum_{ k = 1}^n  || \phi_k||_{\infty}   
	   	  															+   \sum_{k = 1}^n ||\phi_k||_{\infty}^2.
	   $$
	Here and in $(ii)$,
	 $K$ is the constant from 
	$
  	    (A2) (p, \kappa).
	$ 
	 Observe that $\bar f(u, t, x)$ and $h (u, t, x)$ satisfy the Assumption 8.6 of \cite{Kr_99}
	with 
	$
		K = \bar K, 
				\xi \equiv \beta, 
						s = \infty$.
	  Hence, the claim  follows from  Theorem 8.5 of \cite{Kr_99} (see also Remark \ref{remark 2.2}).
	
	$(ii)$  For any 
		$
		   h \in \cH (T)
		$, 
		there exists a unique solution 
		$
		   \cR h \in    \fH^{1/2 - \kappa}_p (T)
		$
		   of \eqref{1.2}.
		This time one needs to set
		$$
			\bar f (u, t, x) = f (u, t, x) + u \partial_t h (t, x),  
												   \quad h (u, t, x) = 0,
		$$
		 $$
			 \bar K  = 1 +  K + ||\partial_t h ||_{ B ([0, T] \times \bR) }
		 $$
		and use the argument of $(i)$.

	\end{remark}

		\mysection{Auxiliary results.}
			\label{section 3}
	For 
	 $
		k_1, k_2 \in \bN
	$,
	and
	$
		l_1, l_2 \in \bN \cup \{0\}
	$,	 
	we denote
	$$
		\phi^{l_1, l_2}_{k_1,  k_2} (x): = |\phi_{k_1} (x)|^{l_1} |\phi_{k_2} (x)|^{l_2}.
	$$

	Set
	 $$
		\delta w^i (t, h) : = w^i (t) - w^i (t, h).
	 $$
	  $$
		s^{i j} (t, h) : = \int_0^t  \delta w^i (r, h) \, d_r w^j (r, h) - \delta_{i j} t/2.
	  $$		
	and,
	 for any sequence of positive numbers
	$
	  \{\gamma_n, n \in \bN\},
	$ 
	 we denote
	$$
		 \delta w^i_n (t)  : = \delta w^i (t, \gamma_n), \quad
													s^{i j}_n (t) : = s^{i j} (t, \gamma_n), \, \, 
	$$

		\begin{definition}
			\label{definition 3.1}
	Let 
	$
	  \{\gamma_n, n \in \bN\}
	$ 
	be a sequence of positive numbers,
	and 
	$
	  (\alpha, \beta) \in \{(1, 0), (-1, 1)\}.
	$
	We say that  a function defined on 
	$
	\Omega \times [0, T] \times \bR
	$
	 is of type
	 $\bf{\Delta}_{n} $
	 if it can be represented as 
	  $$
	 	    \sum_{ i, j = 1 }^{n}   c_{i j} \phi_{i, j}^{l_{i j}, m_{i j}} (x))  q_{i j} (t) + \sum_{i = 1}^n d_{i} D^{ k_{i} } \phi_{i} (x) \delta w^i_n (t),
	  $$
	where
	  \begin{enumerate}[label=(\roman*)]

		\item 
		$
		  c_{i j}, d_i 
		 $ are some constants,
		depending only on $\alpha$ and $\beta$,
		such that
		   $
			|c_{i j}|, |d_i| \leq 2
		   $;\\
	   \item $k_{i}, l_{i j}, m_{i j}  \in \{0, 1, 2\}$;\\
	\item $q_{i j} $ is either 
		 $
		   \delta w^i_n
		  $ or
		     $
			s^{i j}_n.
		     $
			\end{enumerate}	

	 In the sequel we denote any function of type $\bf{\Delta}_{n}$ 
		 by 
		 $\Delta_{n}$ without specifying the exact expression of $\Delta_{n}$.
	\end{definition}

	\begin{lemma}
		\label{lemma 3.5}
	Let 
	$\gamma \in \bR$, 
					  $p \geq 2$, 
	     						   $T > 0$
	be numbers,
	 and 
	 $
	  u \in \mathcal{H}^{\gamma}_p (T),
	$
	  $
	    \psi \in C^{\infty}_0.
	  $
	Denote 
	$
	  f = \bD u,
	$
	$
	  g = \{g^k, k \in \bN\} = \bS u.
	$
	Then, the following assertions hold.
	
	$(i)$
	The process
	 $
		(u (t, \star), \psi (\star)), t \geq 0
	$ 
	is a semimartingale. 
	
	$(ii)$ 
	 There exists a set
	 $\Omega'$ 
	of probability $1$ 
	such that, for any
	 $
	\omega \in \Omega',
	$
	 $
	  t \in [0, T],
	 $
	and 
	$k \in \bN$,
	 we have
	\begin{equation}
		\label{3.5.1}
		< (u (\cdot, \star), \psi (\star)), w^k (\cdot)> (t)
													 = \int_0^t (g^k (s, \star), \psi (\star)) \, ds,
	\end{equation}
	where $< , > (t)$ stands for the mutual quadratic variation of two real-valued semimartingales. 
	\end{lemma}

	\begin{proof}
	$(i)$
	 For any $\omega \in \Omega, t \in [0, T]$, we have (see Remark \ref{remark 2.4})
	$$
		(u(t, \star), \psi (\star)) = (u (0, \star), \psi (\star)) + F(t) + G(t),
	$$
	where
	  $$
		F (t) = \int_0^t (f (s, \star), \psi (\star)) \, ds,
	  $$
	    $$
		G (t) = \sum_{k = 1}^{\infty} \int_0^t (g^k (s, \star), \psi (\star)) \, dw^k (s).
	    $$

	First, we show that $F$ has a finite variation on $[0, T]$ a.s.
	It suffices to prove that
	$$
		\hat F : = \int_0^T |(f (s, \star), \psi (\star))| \, ds  < \infty \, \text{a.s.}
	$$
	Recall that by the definition of stochastic Banach spaces 
	$
	  f \in \cH^{\gamma - 2}_p (T).
	$
	 By the properties of $H^s_p$ spaces (see Section \ref{section 2})
	 and H\"older's inequality we have
	 $$
		\hat F \leq \int_0^T \int_{\bR} |[(1 - D^2_{x})^{(\gamma - 2)/2} f (s, x)] [(1 - D^2_{x})^{(2 - \gamma)/2} \psi (x)]| \, dx  \, ds
	 $$
	   $$
			\leq T^{p'} || \psi ||_{ 2 - \gamma, p'}  (\int_0^T || f (s, \star) ||^p_{ \gamma - 2, p})^{1/p} < \infty \, \, \text{a.s.},
	   $$
	where $p' = p/(p-1)$.
	By this  we only need to show that 
	$
	  G (t), t \geq 0
	$
	 is a martingale.

	Next, denote
	$$
		G_n (t) := \sum_{k = 1}^{n} \int_0^{t} (g^k (s, \star), \psi (\star)) \, dw^k (s).
	$$
	We will show that
	\begin{equation}
		\label{3.5.2}
		\lim_{n \to \infty} E \sup_{t \leq T} | G_n (t) - G (t) |^2 = 0, 
	\end{equation}
	and, by this $G$ is a  square integrable continuous martingale.
	 First, using Burkholder-Davis-Gundy inequality,
	 for any 
	$
	  n \in \mathbb{N}
	$, 
	and any  
	$
	  m \in \mathbb{N} \cup \{\infty\}
	$
	 such that $m  \geq n$,
	 we get
		\begin{align*}
		&  E \sup_{t \leq T} |  \sum_{ k = n }^m  \int_0^t  (g^k (s, \star), \psi (\star)) \, dw^k (s) |^2 \\
																					&\leq V_{n, m} := 3 E \sum_{k = n}^m \int_0^T |(g^k (s, \star), \psi (\star))|^2 \, ds.
		\end{align*}
		Second, by repeating the argument of Remark 3.2 of \cite{Kr_99} we obtain
	 \begin{equation}
	   \begin{aligned}
		\label{3.5.3}
		V_{n, m} & \leq        N E \int_0^T || (\sum_{ k = n}^m |(1 - D^2_{x})^{(\gamma - 1)/2}  g^k (s, \star)|^2)^{1/2}||^2_{p} \, ds \\
																										&\leq N T^{(p - 2)/p}  || g ||^p_{\mathbb{H}^{\gamma - 1}_p (T)},	
	   \end{aligned}
	 \end{equation}
	where 
	$
	  N = 3  || (1 - D^2_{x})^{ (1 - \gamma)/2 } \psi ||_1
												 ||  (1-  D^2_{x})^{(1  - \gamma)/2 } \psi ||_{ p/(p-2)}.
	$
	Then,  \eqref{3.5.2} holds,
	and this implies the assertion $(i)$.

	$(ii)$
	Using
	 linearity of mutual quadratic variation and It\^o's formula,
	 for all $\omega$, 
	and $t \in [0, T]$, 
	and $n \geq k$,
	we have
	$$
		< G_{n}, w^k> (t)
 							= 	\int_0^t (g^k (s, \star), \psi (\star)) \, ds.
	$$
	Thus, there exists a set $\Omega'$ of probability $1$ such that,
	for any 
	$\omega \in \Omega'$, 
	and every 
	$ t \in [0, T]$, 
		$k \in \bN$,
	\begin{equation}
		\label{3.5.7}
		<  G, w^k> (t)
					 = \lim_{n' \to \infty} <  G_{n'}, w^k> (t)
				  									      = \int_0^t (g^k (s, \star), \psi (\star))  \, ds,
	\end{equation}
	where the $n'$ is some subsequence.
	Here, the passage to the limit is justified by Kunita-Watanabe inequality and \eqref{3.5.3}.	
	\end{proof}

		\begin{lemma}
		\label{lemma 3.1}
	Assume the conditions of Theorem \ref{theorem 2.1}.
	Take any  sequence 
		$
		  \{\gamma_n, n \in \bN\}
		$
		with positive terms,
		 and let $v_n$ be the unique solution
			 of class 
			$
			\cH^{1/2 - \kappa}_p (T)
			$
			of \eqref{2.16}.
	 Let $h (t, x) = h (\omega, t, x)$
	be a function such that, 
	for any 
	$
	 \omega \in \Omega
	$,
	 $
		h \in C^2_{loc} ([0, T]\times \bR).
	 $
 	Denote
	   $$
		   \xi^{(1)}_n (t, x) =    \alpha \sum_{i=1}^n  \delta w^i_n (t) \phi_{i} (x), 
														  		  \quad \xi^{(2)}_n (t, x) = \alpha^2 \sum_{i, j = 1}^n  s^{i j}_n (t) \phi^{1, 1}_{i, j} (x),
	   $$
	      $$	
		\bar v_n (t, x) = v_n (t, x) -  v(t, x)
								  +   \xi^{(1)}_n (t, x) v_n (t, x)  
																				- \xi^{(2)}_n (t, x) h (t, x).
	      $$	    
		
	Then, (a.s.) for all
	    $ t \in [0, T]$, 
					$\psi \in C^{\infty}_0$,
	 the function $\bar v_n$ satisfies the following equation:
	\begin{equation}\
		\label{3.1.0.1}
		              z (t, x) = \int_0^t ([a (s, \star) D^2_{x} z (s, \star)  + b (s, \star) D_{x} z (s, \star)], \psi (\star) ) \, ds 
	\end{equation}	
	  $$
		+ \sum_{k = 1}^8 \int_0^t (F^{(k)}_{ n} (s, \star), \psi (\star) ) \, ds  + \sum_{k = 1}^3 \sum_{i = 1}^{\infty}  \int_0^t (G^{(k)}_{n, i} (s, \star),  \psi (\star))\, dw^i (s),
	  $$
        where
	      $$
		       F^{(1)}_n (s, x) =    a (s, x) D_x (v_n (s, x)  \Delta_{n} (s, x)), 
	      $$
	       $$
		       F^{(2)}_n (s, x) =  (a (s, x) \Delta_n (s, x) + b (s, x) \Delta_n (s, x)) v_n (s, x),
	       $$
		$$
			 F^{(3)}_n (s, x) =  a (s, x) D^2_{x} ( \xi^{(2)}_n (s, x) h (s, x) ) + b (s, x) D_{x} (\xi^{(2)}_n (s, x) h (s, x)),
	        $$
	         $$	       
		      F^{(4)}_n (s, x) =   f (v_n, s, x)  -  f (v, s, x),
	         $$
		  $$	  
			F^{(5)}_n (s, x) =  \Delta_{n} (s, x) f (v_n, s, x),
		  $$
		    $$		 
			  F^{(6)}_n (s, x) =    \Delta_{n} (s, x) v_n (s, x),
	  	    $$
		     $$
			  F^{(7)}_n (s, x) =      \Delta_{n} (s, x) \partial_s h (s, x),
		     $$
		      $$
			  F^{(8)}_n (s, x) =  \sum_{i, j =  1}^n \alpha^2 (v_n (s, x) - h (s, x)) \phi^{1, 1}_{i, j} (x) D s^{i j}_n (s),
		      $$
			$$
				G^{(1)}_{ n, i } (s, x) =  (\alpha + \beta) (v_n (s, x)  -  v (s, x)) \phi_i (x),  i = 1, \ldots, n,
		 	$$
			 $$
												\quad   G^{(1)}_{n, i} (s, x) = \beta  (v_n (s, x)  -  v (s, x)) \phi_i (x) , i > n,
			 $$
	 		   $$
				G^{(2)}_{ n, i} (s, x) =  0, i = 1, \ldots, n, 
													 \quad G^{(2)}_{ n, i } (s, x) =   - \alpha    v (s, x) \phi_i (x),  i > n,
	  		  $$
	    		   $$         	
		     	 	G^{(3)}_{ n, i } (s, x) =    \Delta_{n} (s, x)  v_n (s, x)  \phi_i (x), i \in \bN.
	     		  $$
		Here, all the functions $\Delta_{n}$ are
		possibly different functions of type ${\bf \Delta_{n} }$.

	\end{lemma}

	\begin{proof}
	For the sake of convenience,
         we omit the dependence of functions on the spatial variable.
	Also, every time a new function of type ${\bf \Delta_n}$ appears,
	we explicitly write the constants 
	$
	  d_i, c_{ i j }, i, j = 1, \ldots, n
	$
	 to demonstrate that
	the condition $(i)$ of Definition \ref{definition 3.1} holds.

	Keeping in mind V. Mackevi\v cius's method (see Section \ref{section 1}),  
	 we subtract \eqref{2.15} from \eqref{2.16} and formally write the 'stochastic' part of $v_n - v$ as follows:
	$$
		\alpha \sum_{i = 1}^n  v_n (t) \phi_i \, dw^i_n (t) + \beta \sum_{i  = 1}^{\infty} v_n (t) \phi_i \, dw^i (t) 
	$$
	 $$
			- (\alpha + \beta) \sum_{i = 1}^{\infty} v (t) \phi_i \, dw^i (t)
	 $$
	  $$
		= \alpha \sum_{ i = 1}^n v_n (t) \phi_i \, d(w^i_n (t) - w^i (t))
	 $$
	   $$
		+ (\alpha + \beta) \sum_{ i = 1}^n (v_n (t) - v(t)) \phi_i \, dw^i (t)
	   $$
	    $$
		 + \beta   \sum_{i = n+1}^{\infty} (v_n (t) - v (t)) \phi_i \, dw^i (t)
	    $$
	     $$
		 - \alpha \sum_{ i = n+1}^{\infty}  v (t) \phi_i \, dw^i (t).
	     $$
	By the above, 
	 for any 
	$
	\psi \in C^{\infty}_0
	$, 
					$t \in [0, T]$,
							   $\omega$, 
	the function  $v_n - v$ satisfies the following equation:
	 \begin{equation}
		\label{3.1.1}
		(v_n (t) - v (t)), \psi) = \sum_{k =1 }^{8} I^{(k)}_n (t),
	   \end{equation}
	where
		$$
		I^{(1)}_n (t)  =  \int_0^t (a (s) D^2_{x} [v_n (s) - v (s)] , \psi)\, ds,
		$$
		 $$
		     I^{(2)}_n (t) = \int_0^t (b (s) D_x [v_n (s) - v (s)], \psi) \, ds,
		 $$
		  $$
		     I^{(3)}_n (t) =  \int_0^t  (f (v_n, s) - f (v, s) , \psi) \, ds,
		  $$
		    $$
		 	I^{(4)}_n (t) =  \alpha \sum_{i  = 1}^n \int_0^t   (v_n (s) \phi_{i}, \psi) \, d(w^{i}_n (s) - w^i (s)),
		    $$
		     $$
		       I^{(5)}_n (t)=  (\alpha + \beta) \sum_{i = 1}^n   \int_0^t  ( [v_n (s)  - v (s)]  \phi_{i}, \psi) \, dw^i (s), 
		     $$
		      $$
		         I^{(6)}_n (t) =  \beta   \sum_{ i= n + 1}^{\infty} \int_0^t  ( [v_n (s) - v (s)] \phi_i, \psi)\,  dw^i (s),
		      $$
		       $$
			  I^{(7)}_n (t) = - \alpha \sum_{i = n + 1}^{\infty} \int_0^t ( v (s) \phi_i, \psi) \, dw^i (s),
		       $$
		            $$
		    	     I^{(8)}_n (t) =  - (\alpha^2/2 + \alpha \beta) \sum_{i = 1}^n \int_0^t ( v_n (s) \phi_{i}^2, \psi) \,ds.
		          $$
	In what follows, all the identities hold a.s., for all $t \in [0, T]$.

	Note that 
	$
	  \phi_i \psi \in C^{\infty}_0,
	$
	 and, then,
	 by Lemma \ref{lemma 3.5} $(i)$
	the process 
	$
	  (v_n (t), \phi_i \psi), t \geq 0
	$
	 is a semimartingale.
	Using integration by parts formula for semimartingales, we get
	\begin{equation}
		\label{3.1.2}
		I^{(4)}_n (t) =  - (v_n (t) \xi^{(1)}_n (t),  \psi) 
		    + I^{(4, 1)}_n (t) + I^{(4, 2)}_n (t),
	 \end{equation}
	where
	  \begin{equation}
		\label{3.1.3}
		I^{(4, 1)}_n (t)  =   \alpha \sum_{ i = 1}^n \int_0^t \delta w^i_n (s) \,d(v_n (s), \phi_i \psi),
	  \end{equation}
	   $$
		 I^{(4, 2)}_n (t) =    \alpha \sum_{i = 1}^n <(v_n (\cdot),  \phi_i \psi),  w^i (\cdot)> (t). 
	   $$
		By Lemma \ref{lemma 3.5} $(ii)$ we have
		\begin{equation}
			\label{3.1.4}
			I^{(4, 2)}_n (t) 	=   \alpha \beta \sum_{ i = 1}^n \int_0^t (v_n (s) \phi_i^2, \psi) \, ds.
		\end{equation}
	
	Next, using  associativity of stochastic integral, we  write
	\begin{equation}
		\label{3.1.5}
		I^{(4, 1)}_n  (t) =  \sum_{k = 1}^6   I^{(4, 1, k)}_n (t),
	\end{equation}
	  where
	  $$
		I^{(4, 1, 1)}_n (t) = \alpha \sum_{i = 1}^{n} \int_0^t  (a (s) D^2_{x} v_n (s),  \phi_i \psi) \delta w^{i}_n (s)  \, ds,
	  $$
           $$
		   I^{(4, 1, 2)}_n (t) = \alpha \sum_{ i  = 1}^n  \int_0^t  (b (s) D_{x} v_n (s), \phi_i \psi) \delta w^{i}_n (s) \, ds,
	   $$
	    $$
	  	  I^{(4, 1, 3)}_n  (t) = \alpha  \sum_{i  = 1}^n \int_0^t  (f (v_n, s), \phi_i \psi) \delta w^{i}_n (s) \, ds,
	    $$
	       $$
			I^{(4, 1, 4)}_n (t)=  \alpha^2 \sum_{ i, j =  1}^n  \int_0^t   (v_n (s), \phi_{i, j}^{1,1}  \psi) \,\delta w^{i}_n (s) \, D w^j_n (s) \, ds,
	       $$
		$$
			I^{(4, 1, 5)}_n (t)= - \alpha (\alpha^2/2 + \alpha \beta) \sum_{ i, j  = 1}^n \int_0^t  (v_n (s), \phi_{i, j}^{1, 2}  \psi) \, \delta w^{i}_n (s) \, ds
		$$
		 $$
			= \int_0^t ( \Delta_{n} (s) v_n (s), \psi) \, ds,
																 \, \, d_i = 0, c_{i j} = -  \alpha (\alpha^2/2 + \alpha \beta), 
		 $$
		  $$
			I^{(4, 1, 6)}_n (t)=  \alpha \beta  \sum_{ j = 1}^{\infty}  \sum_{ i = 1}^n \int_0^t   (v_n (s), \phi_i  \phi_j \psi)  \delta w^{i}_n (s)   \,   dw^j (s)
		  $$
		    $$
			 =  \sum_{ j =  1}^{\infty} \int_0^t (  \Delta_{n} (s) v_n (s) \phi_j, \psi) dw^j (s),  
																			\, \,  d_i = \alpha \beta,  c_{i j} = 0.
		    $$

	In this paragraph we show that in 
	$
	  I^{(4, 1, k)}_n (t), k = 1, 2, 3
	$
	 one may replace each distribution by its product with $\phi_i$. 
	First, note that, for any
	 $\omega, s, i$,
	we have 
	$
	   a(s) \phi_i \in C^{1 + 1/2 + \kappa + \eta}.
	$
	Then, it follows from Lemma 5.2 $(i)$ of \cite{Kr_99} that 
 	$$
		 \phi_i D^2_{x} v_n (\cdot),   a (\cdot) \phi_i D^2_{x} v_n (\cdot)  \in \mathbb{H}^{-3/2 - \kappa}_p (T).
	$$ 
	Further, by the standard approximation argument combined with Lemma 5.2 $(i)$ of \cite{Kr_99},
	the following identity holds (in the sense of distributions):
	$$
		\phi_i D^2_{x} v_n (s) =   D^2_{x} (v_n (s) \phi_i) - 2 D_x (v_n (s) D \phi_i) + v_n (s) D^2 \phi_i.
	$$ 
	Then, we get
	\begin{equation}
		\label{3.1.6}
		I^{(4, 1, 1)}_n  (t) = \sum_{k = 1}^3 I^{(4, 1, 1, k)}_n (t),
	\end{equation}
	where
	 \begin{equation}
		\label{3.1.7}
	  \begin{aligned}
		I^{(4, 1, 1, 1)}_n (t)& = \alpha \sum_{i = 1}^n  \int_0^t (a(s) D^2_{x} [v_n (s) \phi_i], \psi) \, \delta w^i_n (s) \, ds  \\
																									&=  \int_0^t  (a (s) D^2_{x}  [ \xi^{(1)}_n (s)  v_n (s)] ,\psi) \, ds,
	  \end{aligned}
	     \end{equation}
	   \begin{equation}
		\label{3.1.8}
	     \begin{aligned}
		I^{(4, 1, 1, 2)}_n (t) & =  -2 \alpha \sum_{i  = 1}^n \int_0^t   ( a(s) D_x [v_n (s)  D \phi_i], \psi) \delta w^i_n (s) \, ds \\
																								&=\int_0^t ( a (s) D_x [ \Delta_{n} (s) v_n (s)], \psi) \, ds,   
																																				\, \,   d_{i} = -2\alpha, c_{i j} = 0,  
	     \end{aligned}
	    \end{equation}
	\begin{equation}
			\label{3.1.8.1}
	  \begin{aligned}
		I^{(4, 1, 1, 3)}_n (t)& =  \alpha \sum_{ i = 1}^n  \int_0^t (a(s) v_n (s)  D^2  \phi_i, \psi) \delta w^i_n (s) \, ds\\
	  																					& =   \int_0^t ( a (s) \Delta_{n} (s) v_n (s), \psi) \, ds,
																		\, \, d_{i} = \alpha, c_{i j} = 0.
	\end{aligned}
	\end{equation}
	By the same argument we have
	$$
		b (\cdot) \phi_i D_x v_n (\cdot)  \in \bH^{-1/2 - \kappa}_p (T).
	$$
	 Hence, we get
	 $$
		I^{(4, 1, 2)}_n (t) = \int_0^t (b (s) \xi^{(1)}_n (s) D_x v_n (s), \psi) \delta w^i_n (s) \, ds,
	 $$
	and
	  \begin{equation}
		\label{3.1.9}
		I^{(4, 1, 2)}_n (t) = I^{(4, 1, 2, 1)}_n (t) + I^{(4, 1, 2, 2)}_n (t)
	  \end{equation}
	with
	 $$
		I^{(4, 1, 2, 1)}_n (t) = \int_0^t (b (s) D_x  [ \xi^{(1)}_n (s) v_n (s) ], \psi) \, ds,
	 $$
	  $$
		I^{(4, 1, 2, 2)}_n (t) = \int_0^t ( b (s) \Delta_{n} (s) v_n (s), \psi),
															\, \, d_i = - \alpha,  c_{i j} = 0.
	  $$
	We move to 
	$
	  I^{(4, 1, 3)}_n (t).
	$ 
	Note that by 
	$(A2) (p, \kappa) (ii)$
	we have
	 $
	    f (v_n, \cdot) - f (0, \cdot) \in \mathbb{L}_p (T)
	$
	 because 
	$
	  v_n \in C ([0, T], L_p)
	$
	 (see Remark \ref{remark 2.1}).
	Since 
	$
	  f(0, \cdot) \in L_p ([0, T], H^{-3/2 - \kappa}_p),
	$ 
	we have 
	$
	  f(v_n, \cdot) \in \bH^{-3/2 - \kappa}_p (T),
	$
	and, then, by Lemma 5.2 $(i)$ of \cite{Kr_99}  
	the same holds for 
	$
	  f(v_n, \cdot) \phi_i.
	$
	Hence, we may write
	 \begin{equation}
		\label{3.1.10}
	   \begin{aligned}
		I^{(4, 1, 3)}_n  (t) &= \alpha  \sum_{i  = 1}^n \int_0^t  (f (v_n, s) \phi_i, \psi) \delta w^{i}_n (s) \, ds,\\
																						&= \int_0^t (  \Delta_{n} (s) f (v_n, s), \psi) \, ds,
																																\, \, d_i = \alpha, c_{i j} = 0. 
	   \end{aligned}
	   \end{equation}

	Next, observe that
	$$
		 \delta w^i_n (s) D w^j_n (s)  = D s^{i j}_n (s) + \delta_{i j}/2.
	$$
	Then,
	 $$
		I^{(4, 1, 4)}_n (t) = \alpha^2  \sum_{i, j = 1}^n \int_0^t (v_n (s) \phi_{i, j}^{1, 1}, \psi) D s^{i j}_n (s) \, ds  
	 $$
	  $$
		+ \alpha^2/2 \sum_{i = 1}^n \int_0^t (v_n (s) \phi_{i}^2, \psi) \, ds,
	  $$
	so that by this and \eqref{3.1.1} and \eqref{3.1.4} we have
	 \begin{align*}
		 &  R_n (t): = I^{(8)}_n (t) + I^{(4, 2)}_n  (t) +	I^{(4, 1, 4)}_n (t)\\
																	& =   \alpha^2  \sum_{i , j = 1}^n \int_0^t  (v_n (s) \phi_{i, j}^{1,1}, \psi)  \, D s^{i j}_n (s) \, ds.
	 \end{align*}
	In fact, this cancellation is the reason why we have the 'correction term'  
	 $
	  -(\alpha^2/2 + \alpha \beta) v_n (s)  \sum_{i  = 1}^n  \phi_i^2 \, ds
	$
	 on the right hand side of \eqref{2.16}.

		Next, in the integral $R_n (t)$ we split $v_n$ into $v_n - h$ and $h$
		and  integrate by parts in the integral containing $h$.
		Then, we get 
	\begin{equation}
		\label{3.1.14}
					R_n (t)  =   (\xi^{(2)}_n (t) h (t), \psi)    +  R^{(1)}_n (t) + R^{(2)}_n (t),
	 \end{equation}
	where
	 $$
		R^{(1)}_n (t) = - \alpha^2 \sum_{i, j = 1}^n  \int_0^t  (\partial_s h (s) \phi^{1, 1}_{i, j}, \psi)  s^{i j}_n (s) \, ds,
	 $$
	  $$
		= \int_0^t (\Delta_n (s) \partial_s h (s), \psi) \, ds,  \,\,   d_i = 0, c_{i j} =  -\alpha^2,
	 $$
	  $$
		R^{(2)}_n (t) = \alpha^2 \sum_{i, j = 1}^n \int_0^t ([v_n (s)  - h (s)]  \phi^{1, 1}_{i, j}, \psi) D s^{i j}_n (s) \, ds. 
	  $$
	We note that
	 this time
	 the mutual quadratic variation term vanishes because
	$s^{i j}_n$ has a locally bounded variation.

	Finally, we combine all the terms that we got from the integration by parts.
	First, observe that by \eqref{3.1.1}, \eqref{3.1.7} and \eqref{3.1.9}
	\begin{equation}
		\label{3.1.12}
	\begin{aligned}
		& I^{(1)}_n (t) + I^{(4, 1, 1, 1)}_n (t)  +  I^{(2)}_n (t) + I^{(4, 1, 2, 1)}_n (t)  \\
		& = \int_0^t ([a (s) D^2_{x} \bar v_n (s)  + b (s) D_{x} \bar v_n (s)  +  F^{(3)}_n (s) ], \psi)  \, ds.
	 \end{aligned}
	  \end{equation}
	Next,  by the above  
	  we obtain 
	$$
		I^{(4, 1, 1, 2)}_n (t) =  \int_0^t (F^{(1)}_n (s), \psi) \, ds  \, \, (\text{see} \, \eqref{3.1.8}),
	$$
	 $$
		I^{(4, 1, 1, 3)}_n (t) + I^{(4, 1, 2, 2)}_n (t) = \int_0^t (F^{(2)}_n (s), \psi) \, ds  \, \, (\text{see} \, \eqref{3.1.8.1} \,  \text{and} \, \eqref{3.1.9}),
	 $$
	  $$
		I^{(3)}_n (t)      = \int_0^t (F^{(4)}_n (s), \psi) \, ds \, \, (\text{see} \, \eqref{3.1.1}),
	  $$
	   $$
		I^{(4, 1, 3)}_n (t)   = \int_0^t (F^{(5)}_n (s), \psi) \, ds,  \, \, (\text{see} \, \eqref{3.1.10}),
	   $$
	     $$
		 I^{(4, 1, 5)}_n (t)   =	\int_0^t (F^{(6)}_n (s), \psi) \, ds, \, \, (\text{see} \, \eqref{3.1.5}),
	     $$
	 	$$
		     R^{(1)}_n (t)    =  \int_0^t (F^{(7)}_n (s), \psi) \, ds,  \, \, (\text{see} \, \eqref{3.1.14}),
		$$
		  $$
			R^{(2)}_n (t) =  \int_0^t (F^{(8)}_n (s), \psi) \, ds, \, \, (\text{see} \, \eqref{3.1.14}),
		  $$
		   $$
			I^{(5)}_n (t) + I^{(6)}_n (t) = \sum_{i = 1}^{\infty} \int_0^t (G^{(1)}_{n, i} (s), \psi) \, dw^i (s),  \, \, (\text{see} \, \eqref{3.1.1}),
		   $$
		    $$
			I^{(7)}_n (t) =  \sum_{i  = 1}^{\infty} \int_0^t (G^{(2)}_{n, i} (s), \psi) \, dw^i (s), \, \, (\text{see} \, \eqref{3.1.1}),
		    $$
		     $$
			  I^{(4, 1, 6)}_n (t) =  \sum_{i = 1}^{\infty} \int_0^t (G^{(3)}_{n, i} (s), \psi) \, dw^i (s), \, \, (\text{see} \, \eqref{3.1.5}).
		     $$
		\end{proof}

		\begin{lemma}
			\label{lemma 3.4}
		Let $\alpha$ and  $\tilde \alpha$ be numbers
		 such that 
		  $
			 0 < \alpha < \tilde \alpha < 1
		$,
		 and let  $X$ be a Banach space.
		For
		 $\theta \in (0, 1)$,
		   and $t > 0$, we denote
		   $
			V^{\theta}_t  = C^{\theta} ([0, t], X)
		   $.
		    Then, for any 
			$
			  f \in V^{\tilde \alpha}_T
			 $,
			 the function 
			$
			  t \to || f ||_{ V^{\alpha}_t  }
			  $ 
			    is continuous on $[0, T]$.	
		\end{lemma}
		The proof can be found in \cite{Y_18}.

		\mysection{Proof of Theorem \ref{theorem 2.1}}
				\label{section 4}
		Take any  sequence 
		$
		  \{\gamma_n, n \in \bN\}
		$ of positive numbers,
		 and let $v_n$ be the unique solution
			 of class 
			$
			\cH^{1/2 - \kappa}_p (T)
			$
			of \eqref{2.16}.
		Later, we will choose 
		$
		   \{\gamma_n, n \in \bN\}
		$
		 such that the desired convergence holds.

		Fix any  $ R > 0$ and denote
		$$
			\cV (t)  := C^{\theta/2 - 1/p} ([0, t], H^{1/2 - \kappa - \mu}_p),
		$$
		 $$
			\cW (t) :=   C^{\theta/2  - 1/p} ([0, t], C^{1/2 - \kappa - \mu - 1/p} ),
		 $$
		  $$
			\sigma_n := \inf \{t \geq 0: ||  v_n   - v  ||_{ \cV (t) }   \geq 1\},
		  $$
		    $$
			\pi (R) := \inf\{ t \geq 0: || v  ||_{ \cV (t) } \geq R \},
		    $$
		     $$
			\tau_n := \sigma_n \wedge \pi (R) \wedge T.
		     $$
		
		Take any
		 $\tilde \theta$
		 such that
		$ 
		  \mu > \tilde \theta > \theta
		$.
		By Remark \ref{remark 2.1} we have
		 $
			v_n, v \in C^{\tilde \theta/2 - 1/p} ([0, T], H^{1/2 - \kappa - \mu}_p)
		 $,
		for any $\omega$.
		Then, by Lemma \ref{lemma 3.4}
		 the functions
		  $
			t \to || z ||_{ \cV (t)},
		  $
		 for 
		$
		  z = v, v_n - v,
		$
		are  $\cF_t$-adapted processes with continuous sample paths.
		This implies that
		 $\pi (R)$,
		 $\sigma_n$ 
		and $\tau_n$ are stopping times.

		By  the fact that 
		$
			1/2 - \kappa - \mu > 1/p
		$
		and Remark \ref{remark 2.1},
		for any 
		$
			\omega,
		$
		 we have
		\begin{equation}
			\label{4.1}
							\sup_{ t \leq \tau_n }  || z (t, \star) ||_p + || z  ||_{ \cW (\tau_n) }    \leq N, \quad z = v, v_n.
		\end{equation} 
		 In this proof $N$ is a constant independent of $n$ that might change from inequality to inequality.
	
		Let $\rho$ be a nonnegative $C^{\infty}_0$ function
		supported on $(0, 1)$ such that
		$
		  \int \rho (x) \, dx = 1.
		$
		Let
		$
			 \nu = (\theta/2 - 1/p) \wedge (1/2 - \kappa - \mu - 1/p),
		$
		and 
		$
		  \varepsilon = 3/\nu.
		$
		We set
		$$
			 \tilde v_n (t, x) = n^{2\varepsilon} \iint_{ [0, 1]\times [0, 1]}  v_n (t - s, x - y) \rho ( n^{\varepsilon} s) \rho ( n^{\varepsilon} y) \, ds dy,
		$$  
		where it is assumed that
		   $
			v_n (t, x) = 0,
		   $
		for
		$ t < 0, x \in \bR $.
		It follows that 
		$
		  \tilde v_n (t, \star), t \in [0, T]
		 $ 
		is a predictable $L_p$-valued function. 
		Moreover, using a change of variables, Young's inequality and \eqref{4.1},
		 we have,
		 for  
		$
		  i, j \in \bN \cup \{0\}, 
		$
		\begin{equation}
					\label{4.0}
\			E \int_0^{\tau_n} || \partial_t^i D^{j}_x \tilde v_n (t, \star) ||_p^p \, dt \leq N n^{ \varepsilon  ( i + j) }.   
		\end{equation}

		\textit{Step 1}.
		We will use Lemma \ref{lemma 3.1} with $h =  \tilde v_n$.
		By 
		the a priori estimate from Theorem 5.1 of \cite{Kr_99} with
		 $
			m: = -3/2 - \kappa
		$
		  and Lemma \ref{lemma 3.1} we have
		\begin{equation}
			\label{4.2}
			||\bar v_n||^p_{\cH^{1/2 - \kappa}_p (\tau_n)}
													\leq N \sum_{k = 1}^{13} I_{ k, n },
		\end{equation}
		where
		$$
			I_{ j, n }  = E \int_0^{\tau_n}  || G^{(j)}_n (t, \star) ||^p_{ m+1, p } \, dt, \, \,    j = 1, 2, 3, 
		$$
		  $$
			I_{4, n }  = E \int_0^{\tau_n} || a (t, \star) D_x ( \Delta_n (t, \star) v_n (t, \star) ) ||_{m, p}^p \, dt,
		  $$
	   	    $$
			I_{5, n}  = E \int_0^{\tau_n} || a (t, \star) \Delta_n (t, \star)  v_n (t, \star)  ||_{m, p}^p \, dt,
		    $$
		     $$
			 I_{6, n}  = E \int_0^{\tau_n} || b (t, \star) \Delta_{n} (t, \star) v_n (t, \star)  ||_{m, p}^p \, dt,
		     $$
		      $$
			 I_{7, n}  = E \int_0^{\tau_n} || a (t, \star) D^2_{x} (\Delta_n (t, \star) \tilde v_n (t, \star)) ||_{m, p}^p \, dt, 
		      $$
		       $$
			  I_{ 8, n }  = E \int_0^{\tau_n} || b(t, \star) D_x (\Delta_n (t, \star) \tilde v_n (t, \star))  ||_{m, p}^p \, dt,
		       $$
		        $$
				I_{ 9, n }  = E \int_0^{\tau_n} || f (v_n, t, \star) - f(v, t, \star) ||_{m, p}^p \, dt,
		        $$
		         $$
				I_{ 10, n }  = E \int_0^{\tau_n} || \Delta_n (t, \star)  f (v_n, t, \star)  ||_{m, p}^p \, dt,
		         $$
		          $$
			       I_{ 11, n }  = E \int_0^{\tau_n} || \Delta_n (t, \star) v_n (t, \star)  ||_{m, p}^p \, dt,
		          $$
			   $$
				I_{12, n}  = E \int_0^{\tau_n} || \Delta_n (t, \star)  \partial_t \tilde v_n (t, \star)  ||_{m, p}^p \, dt,
			   $$
			    $$
				I_{13, n}  = \sum_{i, j = 1}^n n^{ 2 p - 2} E \int_0^{\tau_n} ||    (v_n (t, \star) -  \tilde v_n (t, \star))  \phi^{1, 1}_{i, j} (\star) ||_{m, p}^p \,  |D s^{i j}_n (t)|^p \, dt. 
			    $$
		Here, 
		 $
	           G^{(j)}_n = \{ G^{(j)}_{n, i}, i \in \bN\},  j = 1, 2, 3
		$ 
		are the functions defined in the statement of Lemma \ref{lemma 3.1},
		$
		  {\bf \Delta_{n}},
		$
		and all $\Delta_n$ are possibly different functions of type ${\bf \Delta_n}$
		(see Definition \ref{definition 3.1}).

		\textit{Step 2}.
		 First, by the estimate from Lemma 8.4 of \cite{Kr_99}
		\begin{equation}
			\label{4.3}
			I_{ 1, n } \leq N  E \int_0^{\tau_n} || v_n (t, \star)  -  v (t, \star) ||^p_p  \, dt.
		\end{equation}

		Next, 
		let $\cK$ be the function such that 
		 $$
		      (1 - D^2_{x})^{ (m+1)/2 } z = \cK \ast z,\, \, \forall z \in L_p, 
		$$ 
		where $\ast$ stands for convolution.
		It turns out 
		(see Section 12.9 of \cite{Kr_08})
		that
		$$
			\cK (x)  = \chi |x|^{- (1 - 2 \kappa)/2} \int_0^{\infty} t^{ - (5 - 2 \kappa)/4 } e^{- tx^2 - 1/(4t)} \, dt.
		$$
		where $\chi > 0$ is a constant.
		Then, we have
		 $$
			|(1 - D^2_{x})^{(m+1)/2} G^{(2)}_n (t, x)|^2_{l_2}
														 = \alpha^2 \sum_{ i = n + 1 }^{\infty}  (\int \cK (x - y) v (t, y) \phi_i (y) \, dy)^2. 
		 $$
		Again, by  Lemma 8.4 of \cite{Kr_99} 
		and \eqref{4.1},
		for any $\omega$,  
		and 
		$
		  t \in [0, \tau_n],
		$
		 we have
		$$
			|| G^{(2)}_n (t, \star) ||_{m+1, p} \leq N || v (t, \star) ||_p \leq N.
		$$
		Hence, by the dominated convergence theorem we obtain
		\begin{equation}
			\label{4.4}
			\lim_{n \to \infty} I_{ 2, n }  = 0.
		\end{equation}
 	Using  Lemma 8.4 of \cite{Kr_99} and \eqref{4.1} once more, we get
	\begin{equation}
		\label{4.5}
	  \begin{aligned}
		I_{ 3, n }  &
				  \leq E \int_0^{\tau_n} || \Delta_n (t, \star) v_n (t, \star) ||_p^p \, dt\\
																		&\leq E \int_0^{\tau_n} || \Delta_n (t, \star) ||_p^p \, dt \leq N n^{ N }  \gamma_n^{p/2},
	    \end{aligned}
	  \end{equation}
	where the last inequality is due to Lemma \ref{lemma 3.2} $(i)$.

	\textit{Step 3}. 
	We move to the terms 
	 $
	    I_{ 4, n }  - I_{6, n} 
	  $, 
	  $
		I_{ 11, n }. 
	  $
	First,  by the second inequality in \eqref{4.5}
	\begin{equation}
		\label{4.6}
		I_{ 11, n }  \leq N n^N \gamma_n^{p/2}.
	\end{equation}

	Next, using Lemma 5.2  $(i)$ of \cite{Kr_99}, we get
	$$
		I_{4, n}   + I_{5, n} + I_{6, n} 
									\leq N E \int_0^{\tau_n} || \Delta_n (t, \star) v_n (t, \star)  ||^p_{m + 1, p} \, dt.
	$$
	We point out that before applying Lemma 5.2 to 
	$I_{6, n}$
	one needs to replace 
	$m$ by $m+1$.
	Finally,  we  replace $m+1$ by $0$
	 in the last inequality and  use the second inequality in \eqref{4.5}.
	 We obtain
	   \begin{equation}
		\label{4.7}
		I_{4, n} + I_{ 5, n } + I_{ 6, n }
								\leq N n^{ N}  \gamma_n^{p/2}.
	   \end{equation}

	\textit{Step 3.} We handle the terms 
	$
	  I_{ 9, n } 
	 $ and
	   $
		 I_{10, n}. 
	   $
	 Recall that by Remark \ref{remark 2.1} 
	  $
	   v_n, v \in C([0, T], L_p)
	   $,
	  for any 
	$\omega$.
	By this and
	 $(A2) (p, \kappa) (ii)$  we get
	\begin{equation}
		\label{4.8}
		I_{9, n}  \leq
						    N E \int_0^{\tau_n} || v_n (t, \star) - v (t, \star) ||_p^p \, dt.
	\end{equation}

	Next,  by Lemma 5.2 $(i)$ of \cite{Kr_99} 
	  we have
	$$
		I_{10, n}  \leq N E  A_n  B_n,
	$$
	where 
	$$
		A_n = \sup_{t \leq \tau_n} || \Delta_n (t, \star) ||^p_{C^{1 + 1/2 + \kappa + \eta}},
	$$
	 $$
		B_n = \int_0^{\tau_n} || f ( v_n , t , \star) ||^p_{m, p} \, dt,
	$$
	and $
		 \eta \in (0, 1/2 -\kappa).
	       $
	Fix  any
	 $\delta_1 \in (0, p/2)$. 
	Then, by Lemma \ref{lemma 3.2} $(ii)$ 
	$$
		E A_n \leq N n^{N } \gamma_n^{p/2 - \delta_1}.
	$$
	Splitting
	 $
	  f(u, t, x)
	  $ 
	by
	   $
		f (0, t, x)
	   $ and
	    $
	      f(u, t, x) - f (0, t, x)
	    $
	 and using 
	     $
	       (A2) (p, \kappa) 
	    $,  
	we get that, 
	for any 
	$
	 \omega,	
	$
	and  $t \in [0, \tau_n]$,
	 $$
		 B_n \leq N \int_0^{\tau_n} (||  v_n  ( t , \star) ||^p_{p} + || f (0, t, \star) ||^p_{m, p}) \, dt \leq N,
	 $$
	where the last inequality is due to \eqref{4.1}.
	Hence,
	\begin{equation}
		\label{4.9}
		I_{10, n}  \leq N n^{N}   \gamma_n^{p/2 - \delta_1}.
	\end{equation}

	\textit{Step 4.}
	We deal with 
	  $I_{ 7, n } $,
	    $ I_{ 8, n }$,
	     $I_{ 12, n} $,
		$I_{ 13, n }$.
	First,  by Lemma 5.2 $(i)$ of \cite{Kr_99}
	 $$
		I_{7, n} \leq
					         N E \int_0^{\tau_n} ||\Delta_n (t, \star) \tilde v_n (t, \star) ||^p_{m + 2, p} \, ds.
	 $$
	We replace 
	$
	 m + 2 = 1/2 - \kappa
	$ 
	by $1$, 
	apply Lemma \ref{lemma 3.2} $(ii)$
	and use \eqref{4.0}.
	  We get
	  \begin{equation}
			\label{4.10}
	   \begin{aligned}
		I_{7, n}  & \leq 
		 			 N   E \int_0^{\tau_n} || \Delta_n (t, \star)||_{C^1}   (|| \tilde v_n (t, \star) ||_p^p \\ 
																					&+ || D_x \tilde v_n (t, \star)||_p^p) \, dt
	   																												\leq  N  n^{N} \gamma_n^{p/2 - \delta_1}.
	     \end{aligned}
	    \end{equation}
	By a similar argument
	  \begin{equation}
	           \label{4.11}
	   I_{8, n} + I_{12, n} \leq N  n^{N} \gamma_n^{p/2 - \delta_1}.
	 \end{equation}
	As in Step 2, before applying Lemma 5.2 $(i)$ to $I_{8, n}$
	one needs to replace $m$ by $m+1$.

	Next, 	
	note that,  
	for any 	
	$
	  t \in [0, \tau_n],
	$
	and
	$
	  				s, y \in (0, 1),
	$
	due to \eqref{4.1}
	we have
	$$
		|v_n (t - s/n^{\varepsilon}, x - y/n^{\varepsilon}) - v_n (t, x)| \leq n^{ -\nu \varepsilon } \, || v_n ||_{ \cW (\tau_n) } \leq N n^{ -\nu \varepsilon}.
	$$
	Then, by this  
	$$
		I_{13, n} \leq  n^{2p - 2 - \nu \varepsilon p }   \sum_{i, j = 1}^n ||\phi^{1, 1}_{i, j}||_p^p  \, E  \int_0^{\tau_n}  | D s^{i j}_n (t)|^p \, dt.
	$$
	By Lemma 1.5.2 of \cite{Th_93}, for all $i, j \in \bN$,  
	$$ 
		||\phi^{1, 1}_{i, j}||_p^p  \leq N,
	$$
	and this combined with Lemma \ref{lemma 3.3} $(iv)$ yields
	\begin{equation}
		\label{4.12}
			I_{13, n} \leq N  n^{  (2 - \nu \varepsilon) p } \leq N n^{-p},
	\end{equation}
	because $\varepsilon = 3/\nu$.

	\textit{Step 5}.
	Combining \eqref{4.2} - \eqref{4.12}, we obtain
	    \begin{align*}
		|| \bar v_n ||_{\mathcal{H}^{1/2 - \kappa}_p (\tau_n) }  
													&\leq N  I_{2, n}   +  N n^{- p }  + N n^N  \gamma_n^{p/2 - \delta_1}\\
	 																											&+ N E \int_0^{\tau_n} || v (t, \star) - v_n (t, \star) ||_p^p \, dt.
	    \end{align*}
	Clearly, the above inequality holds with $\tau_n$ replaced by $t \wedge \tau_n$,
	 for any $ t \in [0, T]$,
	and with $N$ independent of $n$ and $t$.
	Then, by Remark \ref{remark 2.1}, for any $t \in [0, T]$,
	  we have
	\begin{equation}
	     \label{4.13}
	   \begin{aligned}
		& E  ||  v_n   - v   ||^p_{ \cV (t \wedge \tau_n) }  \leq
		  N I_{2, n}  + N n^{-p} +  N n^{N}   \gamma_n^{p/2 - \delta_1} 
															+  N E || \Delta_n v_n  ||^p_{ \cV (\tau_n) } \\
		&+  N E || \Delta_n \tilde v_n ||^p_{\cV (\tau_n) }+ N E \int_0^{t } || v_n - v  ||_{ \cV (s \wedge \tau_n) }^p \, ds.
	    \end{aligned}
	  \end{equation}
	Note that  by Lemma \ref{lemma 3.4} the last integral is well-defined.

	Next,
	by the product rule inequality in H\"older spaces and Cauchy-Schwartz inequality we have
	$$
		E || \Delta_n z  ||^p_{ \cV (\tau_n) } \leq |E || \Delta_n ||^{2p}_{ \cV (\tau_n) }   E || z ||^{2p}_{ \cV (\tau_n)} |^{1/2}, \quad  z=  v_n, \tilde v_n.
	$$
	Let us fix some
	  $
	    \delta_2 < (0, 1/2 - \theta/2 + 1/p).
  	  $
	Then, by  Lemma \ref{lemma 3.2} $(iii)$  we get
	$$
		E || \Delta_n ||^{2p}_{ \cV (\tau_n) }  \leq N n^N \gamma_n^{ 2\delta_2 p }. 
	$$

	  Recall that  
	$$
	      ||v_n||_{\cV (\tau_n)} \leq N, \, \, \forall \omega.
	$$
	Using this and  Minkowski inequality for Bochner integral,
	we conclude that
	$$
		 E || \tilde v_n ||^{2p}_{ \cV (\tau_n)} \leq N.
	$$

	 Next, we combine the estimates from the previous paragraph with \eqref{4.13},
	 and we obtain
	$$
		E  ||  v_n    - v   ||^p_{ \cV (t \wedge \tau_n) } 
											\leq  N I_{2, n}  + N n^{-p} +  N n^{N}   (\gamma_n^{p/2 - \delta_1} + \gamma_n^{ \delta_2 p})
	$$
	 $$
			+ N  \int_0^{t } E || v_n  - v  ||^p_{ \cV (s \wedge \tau_n) } \, ds, t \in [0, T].
	$$
	By Gronwall's inequality
	$$
		E  ||  v_n   - v  ||^p_{ \cV (\tau_n) }  
									\leq  N I_{2, n}  + N n^{-p} +  N n^{N}   (\gamma_n^{p/2 - \delta_1} + \gamma_n^{ \delta_2 p}).
	$$
	Since $N$ is independent of $n$, 
	we can choose 
	$
	\{\gamma_n, n \in \bN\}
	$
	 such that 
	$$
		\lim_{n \to \infty}   n^{N}   (\gamma_n^{p/2 - \delta_1} + \gamma_n^{ \delta_2 p})   = 0.
	$$
	By this and \eqref{4.4} we obtain
	\begin{equation}
		\label{4.14}
	   \lim_{n \to \infty} E  ||  v_n   - v  ||^p_{ \cV (\tau_n) }   = 0.
	  \end{equation}

	Next,
	denote 
	$
	  D_n = \{ \sigma_n < \pi (R) \wedge T\}.
	$ 
	Observe that by Lemma \ref{lemma 3.4},
	for any $\omega$,
	 $$
		|| v_n - v ||_{ \cV (\sigma_n) } = 1.
	 $$
	This combined with \eqref{4.14} yields
	\begin{equation}
		\label{4.15}
		 P (\sigma_n \leq \pi (R) \wedge  T) =   E  ||  v_n   - v  ||^p_{ \cV (\tau_n) } I_{D_n}  
																			\leq E || v_n - v||^p_{ \cV (\tau_n) } \to 0		
	 \end{equation}
	as $n \to \infty$.

	\textit{Step 6.} 
	Fix any $\varepsilon > 0$. 
	Since $v \in \mathcal{H}^{1/2  - \kappa}_p (T)$,
	 by Remark \ref{remark 2.1}
	there exists $R > 0$ such that
	\begin{equation}
		\label{4.16}
		P (\pi (R) < T) \leq \varepsilon.
	\end{equation}

	Next, for any $ c > 0$,
	$$
		P (|| v_n - v||_{ \cV (T) } \geq c) \leq P ( || v_n - v ||_{ \cV (\tau_n)} \geq c) 
	$$
	 $$
				+ P ( \pi (R) < T) + P (\sigma_n < \pi (R) \wedge T).
	 $$
	   By Chebyshov's inequality and \eqref{4.14} 
	 $$
		P ( || v_n - v ||_{ \cV (\tau_n)} \geq c) \leq c^{-p} E || v_n - v||^p_{ \cV (\tau_n) } \to 0 
	 $$
	as $n \to \infty$.
	By this and \eqref{4.15}, and \eqref{4.16},
	 for any 
	$c, \varepsilon > 0$
	 we get
	  $$
		\nlimsup_{n \to \infty}  P (|| v_n - v||_{ \cV (T) } \geq c) \leq \varepsilon,
	  $$
	and this finishes the proof of the theorem.

	\mysection{Proof of Theorem \ref{theorem 2.2}.}
		\label{section 5}
	We follow the proof of Theorem 2.9 of \cite{Y_18} very closely,
	making only a few necessary changes.
	
	{\it Proof of the inclusion 
	$
		\text{supp} \, P \circ u^{-1}|_{\cV (T)} \subset \mathfrak{R}_{cl}
	$. }
	By Theorem \ref{theorem 2.1}  one can choose a sequence 
	$
	  \{\gamma_n, n \in \bN\}
	$ 
	such that
	if we denote
	$$
		h_n (t, x) =\sum_{ k = 1}^n (w^k_n (t) \phi_k (x) - t/2 \, \phi^2_k (x)),
	$$
	then
	 $$
		|| \cR (h_n) - u ||_{ \cV (T) } \to 0
	 $$
	as $n \to \infty$ in probability.
	Indeed, set
	 $
	   \alpha  = 1, \beta = 0
	 $.
	For any sequence of positive numbers  
	$
	 \{\gamma_n, n \in \bN\}
	$,
	 we denote by $v_n$  the unique solution 
	of class 
	$
	 \cH^{1/2 - \kappa}_p (T)
	$
	 of \eqref{2.16},
	and by $v$ -- the unique solution of class
	 $
	  \cH^{1/2 - \kappa}_p (T)
	 $
	 of \eqref{2.15}.
	Then, we have 
	$
	  \cR (h_n) \equiv v_n,
	 $ 
	$v \equiv u$
	 as elements of
	 $
	   \cH^{1/2 - \kappa}_p (T)
	 $.

	Next,  by
	  Portmanteau theorem 
	$$
		1 = \nlimsup_{n \to \infty}
			 P ( \cR (h_n) \in \fR_{cl})  \leq P \circ u^{-1}|_{ \cV (T) } (\fR_{cl}),
	$$
	and this yields the desired inclusion.

	{\it Proof of the inclusion   $\fR_{cl} \subset
			 \text{supp} \, P \circ u^{-1}|_{ \cV  (T)} $. }
	Fix any $h \in \cH (T)$.
	For any sequence of positive numbers
	$
	 \{\gamma_n, n \in \bN\}
	$,
	consider the following SPDE:
	\begin{equation}
		\label{5.1}
		dz (t, x) = [a(t, x) D^2_{x} z (t, x) + b(t, x) D_x z (t, x) 
	\end{equation}
	  $$
				 + f (z, t, x)  + z (t, x) \partial_t h (t, x) 
	  $$
	   $$ 
				 - \sum_{k = 1}^n z (t, x) \phi_k (x) \, D w^k_n (t)
														 + 1/2 \sum_{k = 1}^n z(t, x) \phi^2_k (x) ]\, dt
	   $$
	     $$
																									+ \sum_{k = 1}^{\infty}  z(t, x) \phi_k (x) \, dw^k (t),
		  \quad z (0, x) = u_0 (x).
	     $$

	We set   
	$
	 \alpha  = - 1, \beta = 1
	$
	 and  consider the equations \eqref{2.15} and \eqref{2.16} with
	$
	 f(z, t, x)
	$
	 replaced by 
	 $$
		\hat f (z, t, x) = f (z, t, x) + z \partial_t h (t, x),  \quad z, x \in \bR, t \in \bR_{+}. 
	 $$
	Note that 
	$
	   \hat f (z, t, x)
	$
	 satisfies the assumption $(A2) (p, \kappa)$ 
	because
	 $ 
	   \partial_t h  \in B ([0, T] \times \bR).
	 $
	Then, the equation \eqref{2.15} 
	has a unique solution 
	$\hat v
			\in
	  				  \fH^{1/2 - \kappa}_p (T)
	 $
	(see Remark \ref{remark 2.3} $(ii)$),
	and 
	\eqref{2.16}  has a unique  solution
	 $\hat v_n$ 
	of class 
	 $
	  \cH^{1/2 - \kappa}_p (T)
	  $ 
	(see Remark \ref{remark 2.3} $(i)$).
	Observe that 
	$\hat v_n$
	 satisfies the equation 
	\eqref{5.1}.
	 Also note that 
	$\hat v$
	 solves \eqref{1.2},
	and, hence, 
	$
	\hat v \equiv \cR (h)
	$.
	In what follows,  
	 $
	 \{\gamma_n, n \in \bN\}
	 $
	is
	such that, 
	for any 
	$\varepsilon > 0$,
	there exists 
	$
	 N (\varepsilon)  > 0
	$
	such that, 
	for any
	 $
	   n > N (\varepsilon)
	$, 
	\begin{equation}
				\label{5.3}
   	  P (|| \hat  v_n -  \cR (h)||_{ \cV (T) } \leq \varepsilon) > 0. 
	\end{equation}
	The existence of such sequence follows from  Theorem \ref{theorem 2.1}.

	Next, denote
	$$
		h_n (t, x) :=  \sum_{k = 1}^n (w^k_n (t) \phi_k (x)   - t/2  \, \phi_k^2 (x)) - h (t, x),
	$$
	 $$
		W_n (t): = W (t) - h_n (t, \star),
									\quad     \bar w^{k, n} (t) = (W_n (t), \phi_k (\star))_{L_2},
	 $$ 
	and let $P_n$ be a measure on 
	$
	 (\Omega, \cF)
	$ 
	defined by
	 $$
		dP_n = \exp (\int_0^T (\partial_t h_n (t, \star), dW (t))_{L_2}  - 1/2 \, \int_0^T || \partial_t h_n (t, \star)||^2_{2} \, dt) dP.
	 $$

	Next, we claim that
	$$
	       		E \exp ( 1/2 \, \int_0^T || \partial_t h_n (t, \star)||^2_{2} \, dt) < \infty.
	$$
	This easily follows from the fact that 
	$
		h \in B ([0, T] \times \bR),
	$ 
	and 
	$
	  |D w^k_n (t)| < \gamma_n^{-1},
	$
	for any $k, n, \omega$, 
	and $t \geq 0$.
	Then, by Proposition 10.17 of \cite{DPZ_14} Girsanov's theorem is applicable,
	and, then,
	$
	 \{\bar w^{k, n} (t), t \in [0, T], k \in \bN\}
	$
	 is a sequence of independent
	$\cF_t$-adapted 
	standard
	Wiener processes on $(\Omega, \cF, P_n)$.

	For 
	$
	  \gamma \in \bR,
	$
	 we set
	 $
	  \cH^{\gamma}_{p} (T, n)
	 $ 
	to be a stochastic Banach space 
	defined on  
	$
	  (\Omega, \cF, P_n)
	 $ 
	 with
	$
  		\{  w^k (\cdot), k \in \bN\}
	$
	replaced by 
	$
	 \{   \bar w^{k, n} (\cdot), k \in \bN\}
	$.
	Then, $ \hat v_n$ is a unique solution of class 
	$
	  \cH^{1/2 - \kappa}_p (T, n)
	$
	of the following SPDE:
	\begin{equation}
				\label{5.6}
	 \begin{aligned}
		dz (t, x) &= [a (t, x) D^2_{x} z (t, x) + b (t, x) D_x z (t, x) \\  
		& + f (z, t, x)] \, dt  + \sum_{k = 1}^{\infty} z (t, x) \phi_k (x) d\bar w^{k, n} (t),
																		 \quad z (0, x) = u_0 (x).
	 \end{aligned}
	\end{equation}
	Claim that 
	  \begin{equation}
				   \label{5.7}
							P_n \circ \hat v_n^{-1}|_{\cV (T) } = P \circ u^{-1}|_{ \cV (T) }.
	   \end{equation}
	To show this we use Theorem \ref{theorem 2.1}.
		First, there exists  
		a unique solution 
		$
			g \in \fH^{1/2 - \kappa}_p (T)
		$
		 of 
		$
			\partial_t g (t, x) = D_x^2 g (t, x)
		$
		with initial condition $u_0 (x)$.
		By subtracting $g$ from $u$,
		we may assume that 
		$u_0 \equiv 0$.
		Second, by Theorem \ref{theorem 2.1}
		  one may replace the equations \eqref{2.10}  and \eqref{5.6}
		by their Wong-Zakai type approximation schemes 
		(see Definition \ref{definition 2.2}) with mesh size.
		Each Wong-Zakai approximation is a fixed point
		of some contraction operator on 
		$
		  \cH^{1/2 - \kappa}_p (T)
		$  
		 (see, for instance Theorem 5.1 and Theorem 6.3 of \cite{Kr_99}).
		Now, \eqref{5.7} follows from 
		 the embedding theorem for 
		$
		  \cH^{1/2 - \kappa}_p (T)
		$
		(see Remark \ref{remark 2.1})
		and Picard iteration in the space
		 $
		    \cV (T).
		$ 
		
		Finally, we use \eqref{5.3},   
			 the fact that $P_n$ is absolutely continuous with respect to $P$,
		and \eqref{5.7}.
		We obtain that, for any $\varepsilon > 0$,
		$$
			P \circ u^{-1}|_{\cV (T) } ( \{ z \in \cV (T): || z - \cR (h) ||_{ \cV (T) } \leq \varepsilon \} ) > 0.
		$$
		This proves the second inclusion.

	\mysection{Appendix}
		\label{section 6}

	The following lemma is taken from \cite{Y_18}.
	For  readers' convenience the full proof is given here.
	              \begin{lemma}
				 	\label{lemma 3.3}	
		Let $p > 1$, 
		$h \in (0, 1 \wedge T)$
		$\varepsilon > 0$,
		$\theta \in (0, 1/2)$,
		$\theta' \in (0, \theta)$
		 be numbers.
		  Assume that $(A4) (h)$ holds. 
		Then, for any $i, j \in \bN$, the following assertions hold.
		 
			$$
				(i)\,  E || \delta w^i (\cdot, h)   ||^p_{ C [0, T] } 
													 \leq N (p, T, \theta)  h^{   p/2  - \varepsilon }.
			$$

			  $$
				(ii)\,		E ||\delta  w^i (\cdot, h)  ||^p_{ C^{1/2 - \theta } [0, T] } \leq N (p, T, \theta, \theta') h^{   \theta' p}.
			  $$
		
			 $$
				(iii)\,	 E || s^{i j} (\cdot, h)   ||^p_{ C [0, T] } \leq N (p, T, \theta) h^{p/2 - \varepsilon}.
			$$

			\begin{align*}
			 	(iv) \, & E  \int_0^T (|\delta w^i (t, h)  |^p + |s^{i j} (t, h)|^p)  \, dt  \leq N (p, T) h^{p/2},\\
					&	E \int_0^T |D s^{i j} (t, h)|^p \, dt \leq N (p, T).
			  \end{align*}

			$$
				(v)	\, E || s^{i j} (\cdot, h) ||^p_{ C^{1/2 - \theta} [0, T]}  \leq N (p, T, \theta, \theta') h^{  \theta' p}.
			$$

		\end{lemma}

		\begin{proof} 
		Denote  
		$
			  t_k  = k h, k \in \{ -1, 0, 1, \ldots\}.
		$	 
		For any $a > 0$, 
		$
		   f : \bR \to \bR,
		$
		   denote
		$$
			\Delta_a f (x) = f (x + a) -  f(x),
		$$
		 $$
			 \rho_{f} ( h, T) = \sup_{t, s \in [0, T]: |t - s| \leq h} |f (t) - f(s)|.
		 $$
		For the sake of convenience,
		in  the proofs $(i), (ii)$
		we denote 
		$
		   w: = w^i, 
				    w (\cdot, h)  := w^i (\cdot, h). 
		$

		$(i)$
		 For 
		$ 
		  t  \in [t_l, t_{l+1}),
		$
		  we have
		\begin{equation}
				\label{3.3.1}
		  |\delta w (t, h)|  
						\leq  |w(t) - w (t_{l-1})| + \varkappa (\Delta_{h} w (t_{l-1})),
		\end{equation}
		and 
		\begin{equation}
		   \label{3.3.2}			
  					\varkappa (\Delta_{h} w (t_{l-1})) \leq \Delta_{ h } (w (t_{ l - 1 })) + I_{A_{l - 1}},
		\end{equation}
		where 
		$$
			A_l = \{ \Delta_{ h } w (t_{l}) > 1\}.
		$$
		Recall that $w (t) = 0$, for $t \leq 0$.
		By this 
		$
			P (A_{-1}) = 0.
		$

		Next,
		denote 
		$
			M = \lfloor T/h \rfloor. 
		$
		 By Chebyshov's inequality, for any $q > 0$, 
		\begin{equation}
						\label{3.3.3}
								P (\cup_{l = 0}^{M-1} A_{l}) \leq 	\sum_{ l = 0}^{M-1} P (A_l) \leq  N/h \, E |w (h)|^q \leq N h^{ q/2 -  1}.
		\end{equation}
		Then, by this and \eqref{3.3.2}
		\begin{equation}
		\label{3.3.4}
		\begin{aligned}
						& E  \max_{l = 0, \ldots, M}  |\varkappa (\Delta_{h} w (t_{l-1}))|^p  	\\
																		&\leq  N (E \rho^p_{w} (h, T) + P (\cup_{l = 0}^M A_l)) \leq N h^{p/2 - \varepsilon},
		\end{aligned}
		  \end{equation}
		where in the second inequality we used the estimate of $\rho_w$ which we state below.
		   By Theorem 2.3.2 of \cite{Kr_94}, 
		 for any $\alpha > 0$,
		 there exists a positive random variable 
		$
		  N_{\alpha, T}
		$
		 such that,
		for any 
		$
		    r > 0,  \, E N_{\alpha, T}^r < \infty,
		$
		 and
		  \begin{equation}
					\label{3.3.5}
								\rho_{w} (\lambda, T) \leq  N_{\alpha, T} \, \lambda^{ 1/2 - \alpha}, \, \forall \omega \in \Omega,  \lambda \in [0, T].
		\end{equation}
		Then, the claim  follows from 
               \eqref{3.3.1}, \eqref{3.3.4} and \eqref{3.3.5}.
		By the way, similarly, for all $l$, we have
		\begin{equation}
				\label{3.3.6}
			E | \varkappa (\Delta_{h} w (t_l))|^p \leq N h^{p/2}.
		\end{equation}

		$(ii)$
		Fix any
		 $
		    \alpha \in (0, \theta).
		  $
		 First, we consider the case when 
		  $
			 |t - s| \geq h,
		  $
			 $
			    t, s \in [0, T].
			  $
		   We have
		   $$ 
		        1/(t-s)^{1/2 - \theta} | \delta w (t, h)  - \delta w (s, h)| 
			  \leq 2 h^{ \theta - 1/2 } ||\delta w (\cdot, h)  ||_{ C [0, T] },
		    $$
		and this combined with $(i)$ yields the claim.

		Next,
		we take any $t, s \in [0, T]$ such that $|t-s| < h$.
		There are  two subcases: either 
		\begin{equation}
		  \label{3.3.7}
					 (t, s)   \in  B_1 = \cup_{l = 0}^M \{ (t, s) \in [0, T]^2:  t, s \in [t_l, t_{ l+1 }] \}
 		\end{equation}
		or
		 \begin{equation}
		    \label{3.3.8.0}
		   \begin{aligned}
		 			(t, s) \in B_2 = &\cup_{l = 0}^{M-1} \{ (t, s) \in [0, T]^2:	  |t - s| < h, \\
																			&  t_l < s \leq t_{l+1} \leq t < t_{ l + 2} \}.
		   \end{aligned}
		   \end{equation}

		To handle  \eqref{3.3.7} we write 
		  \begin{equation}
				      \label{3.3.8}
                                     |\delta w (t, h)  - \delta w (s, h)| \leq | w (t, h) - w (s, h)| + |w (t) - w (s)|.
                       \end{equation}
		Using \eqref{3.3.5} and the fact that 
		$
			|t - s| \leq h,
		$
		   we get
		\begin{equation}
				\label{3.3.9}
			|w (t) - w (s)|  \leq N_{\alpha, T} h^{  \theta - \alpha } | t -  s|^{1/2 - \theta}.
		\end{equation}                      
	      Next,   
			by   \eqref{3.3.4} we obtain
                          \begin{equation}
					\label{3.3.10}
			  \begin{aligned}
                        E  & \sup_{ (t, s) \in B_1 }      | w (t, h) - w (s, h) |^p  \leq \\
															& |t - s|^p/h^p \,  E  \max_{l = 0, \ldots, M}   |\varkappa (\Delta_{h} w (t_{l - 1}))|^p   
																													     \leq N  h^{ (\theta - \alpha) p} \, |t - s|^{(1/2 - \theta)p}. 
			   \end{aligned}
		       \end{equation}
		       Then, the claim in this subcase follows from \eqref{3.3.8} - \eqref{3.3.10}.

		We move to the second subcase 
		\eqref{3.3.8.0}.
		Observe that
		$$
			   | w (t, h)  -  w (s, h)| 
		$$
		  $$
			\leq  |w (t, h) - w (t_{l+1}, h)|  + |w (s, h) - w (t_{l+1}, h)|,
		  $$
		and 
		$
		   (t, t_{l+1}), (s, t_{l+1}) \in B_1.
		$
		This combined with \eqref{3.3.10} and \eqref{3.3.9} proves the assertion in this  subcase.

		$(iii)$	
		We follow the proof of Proposition 6.3.1 of \cite{S_05}.
		First, we consider the case $i = j$.  By It\^o's formula, for any $t$, a.s. 
		$$
			|w^i (t) - w^i (t, h)|^2 = 2  \int_0^t (w^i (s) - w^i (s, h)) \, d_s (w^i (s)  - w^i (s, h)) + t,
		$$
		and, then,
		$$
			s^{i i} (t, h) = \int_0^t \delta w^i (s, h) \, dw^i (s) - 1/2 \, |w^i (t) - w^i (t, h)|^2.
		$$
 		Using Burkholder-Davis-Gundy inequality and assertion $(i)$,
		we get
		\begin{equation}
		  \begin{aligned}
					\label{3.3.11}
			E || s^{i i} (\cdot, h) ||^p_{ C [0, T] } & \leq N E (|| \delta w^i (\cdot, h) ||^p_{ C [0, T] } \\
																						&	+ || \delta w^i (\cdot, h) ||^{2p}_{ C [0, T] })   \leq N h^{p/2 - \varepsilon}.	
		\end{aligned}
		\end{equation}
		By the same argument
		\begin{equation}
		  \begin{aligned}
					\label{3.3.11.1}
		  \max_{t \leq T} E |s^{i i} (\cdot, h)|^p \leq N h^{p/2}.
		  \end{aligned}
		\end{equation}

		Now we assume $i \neq j$.
		Note that, 
		for 
		$t \in [t_k, t_{k+1}]$,
		we have
		$$
			w^i (t, h) = -\Delta_h w^i (t_{k - 1}) + w^i (t_k) + 1/h (t - t_k) \varkappa(\Delta_h w^i (t_{k-1})).
		$$
		Then, for each $\omega, t$ we may write
		\begin{equation}
					\label{3.3.12}
			s^{i j} (t, h) =  I_1 (t) + I_2 (t)  + I_3 (t),
		\end{equation}
		where
		$$
			I_{1} (t) = \sum_{ l = 0}^{ \lfloor t/h \rfloor  }    \varkappa (\Delta_{h} w^j (t_{l - 1}))   \int_{t_l}^{t_{l+1}}  (w^i (s) -  w^i (t_l))/h  \, I_{s \leq t} \, ds,
		$$
		 $$
			I_2 (t ) =  \sum_{ l = 0}^{ \lfloor t/h \rfloor  }    \varkappa (\Delta_{h} w^j (t_{l - 1}))  \Delta_h w^i (t_{l - 1}),
		$$
		$$
				I_3 (t)  = -  h^{-2}  \sum_{ l = 0}^{ \lfloor t/h \rfloor  }  \int_{t_l}^{t_{l+1}} (s - t_l) I_{s \leq t }  \, ds \, \varkappa (\Delta_{h} w^j (t_{l - 1}))  \varkappa (\Delta_{h} w^i (t_{l - 1})).
		$$
		 Observe that  
		$
		  \varkappa (\Delta_{h} w^i (t_{l - 1}))
		$
		is a symmetric random variable as a composition of an odd function with a symmetric random variable.
		It follows from the Markov property of Wiener process that
		$
		   I_1 (t) 
		$ 
		is a sum of independent centered random variables, and, hence,	by  Doob's maximal inequality
		\begin{equation}
					\label{3.3.13}
			E \sup_{t \in [0, T]}  |I_1 (t) |^p \leq N  E | I_1 (T)  |^p \leq N h^{p/2}.
		\end{equation}
		 Let us explain how to get the second inequality in \eqref{3.3.13}. 
		For 
		$
		   p = \{2, 4, \ldots\},
		$ 
		\eqref{3.3.13} follows from  an elementary combinatorial argument 
		(see, for example, the proof of Lemma 6.3.2 of \cite{S_05}) combined with \eqref{3.3.6}.
		To prove the claim for any $p > 1$, 
			 we pick some 
			$
				k \in \bN \cup \{0\}
			$
			 such that
			$
				2k < p \leq 2k + 2,
			$
			and let $\nu$ be a number determined by the equation
			$$
				1/p = (1 - \nu)/(2k) + \nu/(2k +2). 
			$$
			Then, the desired inequality follows from the assertion for 
			$
			  p = 2k, 2k + 2
			$
			combined
			  with  the log-convexity of $L_p$ norms.
		Further, the same argument
		yields 
		\begin{equation}
					\label{3.3.14}
			E  \sup_{t \leq T} (| I_{2} (t) |^p + |I_3 (t)|^p)  \leq N  h^{p/2}.
		\end{equation}
		The claim $(iii)$ follows from \eqref{3.3.11} - \eqref{3.3.14}.

        $(iv)$	
		First, combining \eqref{3.3.1} - \eqref{3.3.3}, we obtain
		$$
			\sup_{t \leq T} E |w^i (t, h)|^p \leq N h^{p/2},
		$$
		and this implies  the estimate for $\delta w^i (t, h)$.
		By \eqref{3.3.11.1},  \eqref{3.3.13} and \eqref{3.3.14} 
		the above inequality holds with $w^i$ replaced by $s^{i j}$.
		Hence, the first estimate of $(iv)$ is proved.

		Next, by Cauchy-Schwartz inequality 
		$$
			E \int_0^T |D s^{i j} (t, h)|^p \, dt
		$$
		 $$
			 \leq
		     \sum_{k = 0}^{ \lfloor T/h \rfloor } h^{-p} (\int_{t_k}^{t_{k+1}}   E |\delta w^i (t, h)|^{2p}  \, dt)^{1/2}
																						 \, h^{1/2}  \, (E |\varkappa (\Delta_h w^j (t_{k-1}))|^{2p})^{1/2}.
		$$
		This combined with \eqref{3.3.1} and \eqref{3.3.2} proves the second part of claim.

		$(v)$	
			By Cauchy-Schwartz inequality
		we have
	     $$
	   	  E || D s^{i j } (\cdot, h) ||_{ L_{ \infty}[0, T] }^p \leq  h^{-p} M_{1} M_{2},
	     $$
	where 
	      $$
		M_{1} = (E || \delta w^i (\cdot, h)  ||_{  C [0, T] }^{2p})^{1/2}, 
		 \quad 
								      M_{2} = (E  \max_{l = 0, \ldots, \lfloor T/h \rfloor}  |\varkappa (\Delta_{h} w (t_{l-1}))|^{2p} )^{1/2}.
	      $$
	By  $(i)$ and  \eqref{3.3.4}  
	  $$
		M_{1} + M_{2}  \leq N (p,  \varepsilon, T)  h^{ p/2 - \varepsilon}, 
	  $$
   	  Then, by the above
	\begin{equation}
		\label{3.3.15}
					E ||D s^{ i j } (\cdot, h) ||^p_{ L_{\infty} [0, T] } \leq N (p, \varepsilon, T) h^{- 2\varepsilon}.
	\end{equation}
	By the interpolation inequality 
		 (see, for example, Theorem 3.2.1 in \cite{Kr_96}),
		 for any $\lambda > 0$,
		$$
		      || s^{ i j } (\cdot, h) ||_{ C^{1/2 - \theta} [0, T]}
			 \leq  	N (\theta, T) (h^{1/2 + \theta} ||D s^{i j} (\cdot, h) ||_{ L_{\infty} [0, T]} 
				+ h^{ \theta - 1/2 } || s^{i j} (\cdot, h) ||_{ C[0, T]}). 
		$$
		 We finish the proof by combining this with  $(iii)$ and \eqref{3.3.15}.

	  	\end{proof}

			\begin{lemma}
		\label{lemma 3.2}
	Assume that
	 $
	  (A4) (\gamma_n)
	$
	 holds for some sequence
	 $
		\{\gamma_n, n \in \bN\}
	$.
	Let 
	$
	  \theta \in (0, 1), T > 0, p > 1
	$
	 be numbers,
	and let
	 $\Delta_n$ 
	be any function of class 
	${\bf \Delta_n}$.
	Then, the following assertions hold.
	
	$(i)$  
		$$
			 J_n : = E \int_0^T    ||\Delta_n (t, \star) ||^p_p   \, dt \leq N n^{N} \gamma_n^{p/2},
		$$
	where $N =  N (p, T)$.

	$(ii)$  For any $\varepsilon > 0$, 
		   and any 
		$\delta \in (0, 1)$,
		$$
			 E \sup_{t \leq T} || \Delta_{n} (t, \star) ||^p_{C^{2 - \delta}} \leq N n^{N} \gamma_n^{p/2 - \varepsilon},
		$$
		where $N  = N (p, T, \delta, \varepsilon)$.

	$(iii)$ For any $ \varepsilon > 0$, and $\varepsilon' \in (0, \varepsilon)$,
		$$
			  E || \Delta_n ||^p_{ C^{1/2 - \varepsilon} ([0, T], H^{\theta}_p ) } \leq N n^{N} \gamma_n^{\varepsilon' p},
		$$
		where $N = N (p, T, \varepsilon, \varepsilon')$.
		
	\end{lemma}

	\begin{proof}
	$(i)$ Due to Definition \ref{definition 3.1} we have
	\begin{equation}
		\label{3.2.1}
	   \begin{aligned}
		J_n  &\leq N n^{2p - 2}  \sum_{i, j = 1}^{ n}  
										\sum_{ k, l, m = 0}^2  (||D^k_x \phi_i||_p^p +   ||  \phi_{i, j}^{l, m} ||_p^p)  \\
			& \times  E \int_0^T (|s^{i j}_n (t)|^p + |\delta w^i_n (t)|^p) \, dt.
	 \end{aligned}
	\end{equation}
	
	Next, it is well-known (see Lemma 1.5.2 of \cite{Th_93}) 
	that, 
	for any  $\rho \in [2, \infty]$, 
	and $k \in \bN$,
	\begin{equation}
			\label{3.2.4}
											|| \phi_k ||_{\rho}  \leq N (\rho).
	\end{equation}
	In addition (see Section 1.1 of \cite{Th_93}), 
	$$
		D_{x} H_k (x) = 2k H_{k-1} (x).
	$$
	Then, by formula \eqref{2.4}  and what was just said  we have
	\begin{equation}
					\label{3.2.3}
		|| D^k_x \phi_j  ||_{\rho}	 \leq N (k, \rho) j^{k/2}, \, \, 
														 k \in \bN \cup \{0\},   j \in \bN.
	\end{equation}
	Combining \eqref{3.2.1} and \eqref{3.2.3} with Lemma \ref{lemma 3.3} $(iiii)$, we prove the assertion.

	$(ii)$ The proof is similar the one above. 
	First, note that by the interpolation inequality
	for H\"older spaces 
	(see Theorem 3.2.1 in \cite{Kr_96}),
	we may replace $2 - \delta$ by $3$.
	Second, by the product rule
	 \begin{equation}
				\label{3.2.7}
		|| \phi^{l, m}_{i, j}  ||_{ C^3 } \leq N ||\phi_i||^l_{C^3} ||\phi_j||^m_{C^3}.
	\end{equation}
	Third, by Lemma \ref{lemma 3.3} $(i)$, $(iii)$, 
	for any $i, j \in \bN$,
	\begin{equation}
				\label{3.2.8}
		 E ||q_{i j}||_{C[0, T]}^p    \leq N (p, T, \varepsilon) \gamma_n^{p/2 - \varepsilon},
	 \end{equation}
	where 
	 $
	    q_{i j} \in \{ \delta w^i_n, s^{i j}_n \}.
	$  
	Now the assertion follows from \eqref{3.2.7}, \eqref{3.2.3}, and \eqref{3.2.8}.	

	$(iii)$  First, by the properties of 
		$H^{\theta}_p$ spaces, 
		for 
		  $
			k, l, m \geq 0,
		   $
		\begin{equation}
					\label{3.2.5}
		\begin{aligned}
		 ||D^{k} \phi_i||_{\theta, p} +  ||\phi_{i, j}^{ l, m}||_{\theta, p} 
							& \leq  || \phi_i ||_{ k + 1, p } +  ||  \phi_{i, j}^{ l, m} ||_{ 1, p}\\
															& \leq
																N (p) ( ||\phi_i||_{W^{k+1}_p} +  || \phi_{i, j}^{ l, m} ||_{ W^1_p }).
		\end{aligned}
		 \end{equation}
		Clearly,
		 we may assume that
		 $m \neq 0$.
		Note that  by \eqref{3.2.3} the right hand side of \eqref{3.2.5} is less than
		$$
			N (p, k) i^{(k+1)/2}  + N (p) ||  \phi_{i }||^l_{C^1} || \phi_j ||^{m-1}_{ C^1 } (|| \phi_j ||_p + ||D_x \phi_j||_p) 
		$$
	  	 $$
									\leq N (p, k,  l, m) (i^{(k+1)/2} + i^{l/2}  j^{m/2}).
		$$
		This combined with Lemma \ref{lemma 3.3} $(ii)$ and $(v)$ proves the claim.
	\end{proof}

	\end{document}